\documentclass{article}

\usepackage{amsfonts,amsmath,amssymb,amsthm}
\usepackage{graphicx,color}
\usepackage[margin=1in]{geometry}
\usepackage{tikz}
\usepackage{blkarray}
\usepackage{enumitem}

\newtheorem*{theorem*}{Theorem}
\newtheorem{theorem}{Theorem}
\newtheorem{lemma}[theorem]{Lemma}
\newtheorem{proposition}[theorem]{Proposition}
\newtheorem{corollary}[theorem]{Corollary}

\theoremstyle{definition}
\newtheorem{definition}[theorem]{Definition}
\theoremstyle{remark}
\newtheorem{remark}[theorem]{Remark}

\newtheorem{question}[theorem]{Question}
\newtheorem{example}[theorem]{Example}
\newtheorem{claim}[theorem]{Claim}

\numberwithin{theorem}{section}

\newcommand\cA{{\mathcal A}}

\newcommand\cC{{\mathcal C}}
\newcommand\cD{{\mathcal D}}

\newcommand\cH{{\mathcal H}}

\newcommand\cK{{\mathcal K}}
\newcommand\cL{{\mathcal L}}
\newcommand\cM{{\mathcal M}}

\newcommand\FF{{\mathbb F}}

\newcommand\KK{{\mathbb K}}
\newcommand\NN{{\mathbb N}}

\newcommand\RR{{\mathbb R}}

\newcommand\ZZ{{\mathbb Z}}

\newcommand\e{{\bf e}}

\DeclareMathOperator{\cl}{cl}

\DeclareMathOperator{\rank}{rank}

\title{The $k$-fold circuit property for matroids}
\author{Bill Jackson\thanks{School of Mathematical Sciences, Queen Mary University of London, E-mail: \texttt{b.jackson@qmul.ac.uk}} \and Anthony Nixon\thanks{School of Mathematical Sciences, Lancaster University, E-mail: \texttt{a.nixon@lancaster.ac.uk}} \and Ben Smith\thanks{School of Mathematical Sciences, Lancaster University, E-mail: \texttt{b.smith9@lancaster.ac.uk} (corresponding author)}}
\date{\today}

\begin{document}

\maketitle

\begin{abstract}
Double circuits were introduced by Lov\'{a}sz in 1980 as a fundamental tool in his derivation of a min-max formula for the size of a maximum matching in linear matroids. 
This formula was extended to all matroids satisfying the so-called `double circuit property' by Dress and Lov\'{a}sz in 1987. 
We extend these notions to $k$-fold circuits for all natural numbers $k$ and 
show, in particular that several 
families of matroids which are known to satisfy the double circuit property, satisfy the $k$-fold circuit property for all natural numbers $k$. 
These families include all pseudomodular matroids (such as full linear, algebraic and transversal matroids) and certain families of count matroids. These results suggest that the $k$-fold circuit property can be used as a measure of how close the lattice of flats of a matroid is to being a modular lattice. 
\end{abstract}


\paragraph{MSC classes:} 05B35, 06C10, 90C27

\section{Introduction}

Let $\cM=(E,r)$ be a matroid with finite ground set $E$ and rank function $r$. A \emph{circuit} of $\cM$ is a set $C\subseteq E$ such that $r(C)=|C|-1=r(C-e)$ for all $e\in E$. 
Following Lov\'{a}sz \cite{Lov}, we define a \emph{double circuit} of $\cM$ to be a set $D\subseteq E$ such that $r(D)=|D|-2=r(D-e)$ for all $e\in E$. For example, if $\cM$ is the cycle matroid (or graphic matroid) of a graph $G$, then the circuits of $\cM$ are the edge sets of cycles in $G$, and the double circuits are the edge sets of subgraphs of $G$ which are pairs of cycles with at most one vertex in common or \emph{theta graphs}, i.e., graphs consisting of three internally disjoint paths between two vertices. 

In general, a double circuit can contain arbitrarily many circuits and a key insight of Lov\'{a}sz~\cite{Lov} was the concept of the principal partition of a double circuit which exactly describes all of its circuits. More precisely, we can define an equivalence relation $\sim$ on a double circuit $D$ by saying $e\sim f$ if $r(D-e-f)=r(D)-1$. The partition $\{A_1,A_2,\ldots,A_\ell\}$ of $D$ given by the equivalence classes of $\sim$ form the \emph{principal partition} of $D$. It has the property that $C \subset D$ is a circuit of $\cM$ if and only if $C=D-A_i$ for some $1\leq i\leq \ell$.

Double circuits were used
by Lov\'asz~\cite{Lov}  as a tool 
in his celebrated solution of the matroid matching problem for linear matroids.
Subsequently, Dress and Lov\'asz~\cite{DL} extended this solution to all
matroids $M$ that satisfy the \emph{double circuit property}, i.e., 
for every double circuit $D$ in $\cM$, 
\begin{equation}\label{eq:dcpalt}
   r\left(\bigcap_{i=1}^\ell \cl(D\setminus A_i)\right) \geq \ell-2,
\end{equation} 
where $\{A_1,A_2,\dots,A_\ell\}$ is the principal partition of $D$.
In addition, they showed that the double circuit property is satisfied by several families of `full' matroids, including
\emph{full linear matroids}, vector spaces imbued with a matroid structure, and \emph{full algebraic matroids}.
Since every representable (resp. algebraic) matroid can be embedded in a full linear (resp. algebraic) matroid, 
we can efficiently solve the matroid matching problem for these matroids.
Makai~\cite{Mak} extended their results to `full' count matroids, 
i.e. matroids defined on the edge set of a complete multigraph whose independent sets are edge sets of subgraphs satisfying certain sparsity conditions.

A key observation in showing that a full linear matroid has the double circuit property is that its lattice of flats is modular.
This is no longer the case for a full algebraic matroid: instead its lattice of flats satisfies a weaker property called \emph{pseudomodularity} introduced by Bj\"orner and Lov\'asz~\cite{BL}.
Hochst\"attler and Kern~\cite{HK} demonstrated that pseudomodularity is a sufficient condition for a matroid to have the double circuit property.
Moreover, Dress, Hochst\"attler and Kern~\cite{DHK} related pseudomodularity to the problem of finding modular sublattices within the lattice of flats.
In this sense, the double circuit property is closely related to the modularity of a matroid.

\subsection{Main results}

In this paper we generalise the concept of double circuits in a matroid $\cM$ to \emph{$k$-fold circuits}, i.e. sets $D\subseteq E$ such that $r(D)=|D|-k=r(D-e)$ for all $e\in D$, for some fixed integer $k\geq 0$.
Observe that every cyclic set (i.e., a union of circuits) $X$ in $\cM$ is a $k$-fold circuit  when we take $k=|X|-r(X)$. 
The cyclic sets of $\cM$ form a graded lattice whose $k$-th layer consists of the $k$-fold circuits of $M$.
This lattice has already been identified and studied by Tutte~\cite{Tutte} in the 1960's. We will see in Section \ref{sec:prelim} that the lattice of cyclic sets is dual to the well known lattice of flats of $\cM$ so our results on $k$-fold circuits give new insights into this important lattice.

Utilising this lattice structure, we show that the notion of the principal partition of a double circuit can be extended to a $k$-fold circuit $D$ of $\cM$.
Explicitly, we define it to be the partition of  $D$ given by the complements of the $(k-1)$-fold circuits which are contained in $D$.
This principal partition will be a fundamental object for our analysis of $k$-fold circuits.
We also give an alternative characterisation of a connected $k$-fold circuit as a set with an ear decomposition into precisely $k$ ears (Proposition~\ref{prop:kcircuitears}).

Our main results concern a generalisation of the double circuit property which we call the $k$-fold circuit property.
We show that every $k$-fold circuit $D$ with principal partition $\{A_1, \dots, A_\ell\}$ satisfies the inequality
 $$  r\left(\bigcap_{i=1}^\ell \cl(D\setminus A_i)\right) \leq \ell-k$$
and characterise when equality holds (Theorem~\ref{thm:strong}).
We say a $k$-fold circuit $D$ is \emph{balanced} if 
this inequality holds with equality.
Moreover, we say a matroid has the \emph{$k$-fold circuit property} if all of its $k$-fold circuits are balanced.
To see that this is a generalisation of the double circuit property to $k$-fold circuits, observe that 
the above inequality
implies that \eqref{eq:dcpalt} must hold \emph{with equality} in all matroids which have the double circuit property.

We show that full linear matroids have the $k$-fold circuit property for all $k \geq 1$ (Corollary \ref{cor:fulllinear}).
This motivates the $k$-fold circuit property as a measure of how close the lattice of flats of a matroid is to being a modular lattice that is more refined than the double circuit property.
We demonstrate that the family of sparse paving matroids has the $k$-fold circuit property for all $k \geq 3$ (Proposition \ref{prop:sparse+paving}) but may fail the double circuit property. In addition, 
for each $m \geq 2$, we give a family of matroids that do not have the $k$-fold circuit property for all $2 \leq k\leq m$ (Proposition \ref{prop:kfold+failure}).

Showing that classes of matroids satisfy the $k$-fold circuit property requires more subtlety when $k>2$, partly due to the connectivity properties of $k$-fold circuits.
We say a $k$-fold circuit $D$ of $\cM$ is \emph{trivial} if $\cM|_D$ has $k$ connected components and each is a circuit of $\cM$. Double circuits have the property that they are either connected or trivial, the latter of which is very easy to analyse. When $k>2$, there exists non-trivial disconnected $k$-fold circuits, making the analysis much trickier.
To help deal with this, we show that the direct sum of balanced $k_i$-fold circuits is also balanced (Theorem~\ref{thm:disconalt}).

The remainder of the article is dedicated to identifying classes of matroids that satisfy the $k$-fold circuit property for all $k \geq 1$.
We extend Hochst\"attler and Kern's result to show that pseudomodular matroids satisfy the $k$-fold circuit property (Theorem~\ref{thm:pseudomodular+kCP}).
We utilise similar lattice theoretic techniques to~\cite{DHK}, giving local conditions in pseudomodular lattices for finding modular sublattices (Theorem~\ref{thm:modular+substructure}).
This shows full algebraic matroids have the $k$-fold circuit property as a corollary.

We next turn our attention to full count matroids
i.e.~count matroids $\cM(a,b)$ whose underlying graph is a complete multigraph in which each vertex is incident to $\max\{a-b,0\}$ loops, each pair of vertices are joined by $2a-b$ parallel edges and $0\leq b\leq 2a-1$.
We first show that a $k$-fold circuit of $\cM(a,b)$ is balanced if each of its $(k-1)$-fold circuits are sufficiently dense (Proposition~\ref{prop:balanced+Mab}).
This generalises Makai's result \cite{Mak} that full count matroids satisfy the double circuit property,
but does not imply that full count matroids have the $k$-fold circuit property since the
density hypothesis is not satisfied by all their $k$-fold circuits when $k>2$.
Instead we use an alternative approach to establish that the $k$-fold circuit property does hold for the family of full count matroids $\cM(a,b)$ when $1 \leq b \leq a$ (Theorem~\ref{thm:flats}).

The paper is organised as follows. In Section \ref{sec:prelim} we review terminology and preliminary results from matroid theory, 
count matroids and double circuits.
In Section \ref{sec:higher} we define $k$-fold circuits 
and derive their fundamental properties.
In Section \ref{sec:hcp} we generalise the double circuit property of Dress and Lov\'{a}sz to the $k$-fold circuit property.
In Section~\ref{sec:pseudo}, we show that all pseudomodular matroids have the $k$-fold circuit property.
In Section \ref{sec:countM}, we focus on the $k$-fold circuit property for count matroids.
We conclude with brief remarks on possible future directions in Section \ref{sec:conclusion}.

\section{Preliminaries}
\label{sec:prelim}

We will introduce the relevant terminology and preliminary results for matroid theory, 
count matroids 
and double circuits.

\subsection{Matroids and lattices}

We assume some familiarity with the basic notions of matroid theory; 
we refer the reader to \cite{Oxl,Aigner} for basic definitions and concepts.

We briefly recall some fundamentals of lattice theory.
A \emph{lattice} $(\cL, <)$ is a poset such that, for every pair of elements $x$ and $y$ in $\cL$, the least upper bound $x \vee y$ and the greatest lower bound $x \wedge y$ exist.
As such, the join $\vee$ and the meet $\wedge$ are well-defined binary operations on $\cL$.
We shall always assume our lattice has a well-defined minimal element $0 \in \cL$ and maximal element $1 \in \cL$.
A \emph{sublattice} of $\cL$ is a non-empty subset of $\cL$ with the same meet and join operations.

We say \emph{$y$ covers $x$} (and write $x \lessdot y$) if $x < y$ and there exists no $z$ such that $x < z < y$.
The \emph{atoms} of $\cL$ are the elements $a \in \cL$ such that $0 \lessdot a$.
A lattice is \emph{atomic} if every non-zero element can be written as the join of finitely many atoms.

A lattice is \emph{graded} if it is equipped with a map $r\colon \cL \rightarrow \ZZ_{\geq 0}$ satisfying $r(0) = 0$, and if $x \lessdot y$ then $r(y) = r(x) + 1$.
Such a grading sends $x$ to the length of a maximal chain in the interval $[0,x]$.
A graded lattice is \emph{semimodular} if for all $x, y \in \cL$, we have
\begin{equation} \label{eq:semimodular}
r(x) + r(y) \geq r(x \vee y) + r(x \wedge y) \, .    
\end{equation}
We call a pair of elements $(x,y) \in \cL \times \cL$ \emph{modular} if \eqref{eq:semimodular} holds with equality.
We say $\cL$ is \emph{modular} if all pairs of its elements are modular.

A lattice is \emph{geometric} if it is graded, atomic and semimodular.
Every lattice we will encounter will be geometric.
Geometric lattices are directly related to matroids via their \emph{closed sets} or \emph{flats}.
We define the \emph{lattice of flats} $\cL(\cM)$ of a matroid $\cM = (E,r)$ to be $\cL(\cM) = \left\{X \subseteq 2^E \mid \cl(X) = X \right\}$ with partial order  
given by set inclusion.
It is well known that $\cL(\cM)$ is a geometric lattice with meet $X \wedge Y = X \cap Y$ and join $X \vee Y = \cl(X \cup Y)$, and grading given by the matroid rank function $r$.
Conversely, every geometric lattice is isomorphic to the lattice of flats of some matroid~\cite[Theorem 6.1]{Aigner}.

We highlight one more well-known property of flats, see for example \cite[Exercise 1.4.11]{Oxl}, that arises from this lattice theoretic point of view. 
\begin{lemma}\label{lem:covering}
Let $\cM = (E, r)$ be a matroid, $X \in \cL(\cM)$ and
$\{Y_1, \dots, Y_\ell\}$ be the set of flats that cover $X$.
Then $\{Y_1 \setminus X, \dots, Y_\ell \setminus X\}$ is a partition of $E \setminus X$.
\end{lemma}
\noindent
We will refer to this result as the \emph{covering axiom} for flats.
It will be invaluable in our study of $k$-fold circuits.

We next recall a number of ways to decompose and combine matroids.
Given two matroids $\cM_1=(E_1,r_1)$ and $\cM_2=(E_2,r_2)$ with $E_1\cap E_2=\emptyset$, we define their \emph{direct sum} to be the matroid $\cM_1\oplus \cM_2=(E_1\cup E_2,r)$ by putting $r(A)=r_1(A\cap E_1)+r_2(A\cap E_2)$ for all $A\subseteq E_1\cup E_2$.

We can define a relation on the ground set of a matroid $\cM=(E,r)$ by saying that $e,f \in E$ are related if $e=f$ or if there is a circuit $C$ in $\cM$ with $e,f \in C$. It is well-known that this is an equivalence relation.
The equivalence classes are called the \emph{components} of $\cM$.
If $\cM$ has only one component then $\cM$ is said to be \emph{connected}, otherwise 
$\cM$ is said to be \emph{disconnected}.
Moreover, if $\cM$ has  components $E_1, E_2, \dots, E_t$, then 
we can write $\cM$ as $\cM=\cM_1\oplus \cM_2\oplus \cdots \oplus \cM_t$, where $\cM_i=\cM|_{E_i}$ is the matroid restriction of $\cM$ onto $E_i$; see \cite{Oxl}.
We will say that a set $S\subseteq E$ is  \emph{$\cM$-connected} if $\cM|_S$ is a connected matroid. 

We will also consider another way to decompose matroids, namely ear decompositions.
Given a non-empty sequence of circuits $C_1,\dots,C_m$ in a matroid $\cM$, we put $D_i=C_1\cup \cdots \cup C_i$  for all $1\leq i \leq m$ and $\widetilde C_i=C_i \setminus D_{i-1}$ for all $2\leq i \leq m$.
The sequence $C_1,\dots,C_m$ is said to be a \emph{partial ear decomposition} of $\cM$ if, for all $2\leq i \leq m$,
\begin{enumerate}
\item[(E1)] $C_i\cap D_{i-1}\neq \emptyset$,
\item[(E2)] $C_i \setminus D_{i-1} \neq \emptyset$, and
\item[(E3)] no circuit $C_i'$ satisfying (E1) and (E2) has $C_i'\setminus D_{i-1}\subsetneq C_i \setminus D_{i-1}$. 
\end{enumerate}
A partial ear decomposition is an \emph{ear decomposition} of $\cM$ if $D_m=E$.

\begin{lemma}[\cite{C&H}]
\label{lem:ear}
Let $\cM=(E,r)$ be a matroid with $|E|\geq 2$. Then:
\begin{enumerate}
\item $\cM$ is connected if and only if $\cM$ has an ear decomposition.
\item If $\cM$ is connected then every partial ear decomposition is extendible to an ear decomposition.
\item If $C_1,\dots,C_m$ is an ear decomposition of $\cM$ then $r(D_i)-r(D_{i-1})=|\widetilde C_i|-1$ for all $2\leq i \leq m$.
\end{enumerate}
\end{lemma}

A matroid $\cM$ is completely determined by its set of circuits $\cC(\cM)$. We can use this fact to 
define our final two matroid constructions.
Suppose $\cM_1,\cM_2$  are two matroids with ground sets $E_1,E_2$ respectively such that $E_1\cap E_2=\{e\}$ and $e$ is neither a loop nor coloop of $\cM_1$ or $\cM_2$.
The \emph{parallel connection of $\cM_1,\cM_2$ along $e$} is the matroid
$P(\cM_1, \cM_2)$ with ground set $E_1\cup E_2$ and circuits
\begin{align} \label{eq:parallel+connection+circuits}
\cC(P(\cM_1, \cM_2))= \cC(\cM_1)\cup \cC(\cM_2) \cup \{(C_1\cup C_2)-e:e\in C_i\in \cC(\cM_i) \mbox{ for both } i=1,2\}.
\end{align}

\begin{remark}\label{rem:graph+matroid}
We will often consider matroids $\cM$ whose ground set is the edge set of a (multi)graph $G$. In this case we will describe a subgraph $H$  of $G$ by using the properties of its edge set $E(H)$ in $\cM$.
Hence we will refer to $r_\cM(E(H))$ as the {\em rank of $H$ in $\cM$}  and will say that $H$ is {\em $\cM$-independent} or is an {\em $\cM$-circuit} if $E(H)$ is independent or is a circuit in $\cM$. Similarly, we will 
 refer to the subgraph of $G$ induced by  the closure of $E(H)$ in $\cM$ as the {\em $\cM$-closure of $H$} and denote this subgraph by $\cl_\cM(H)$.
We will say that the subgraph  $H$ is {\em $\cM$-connected} if $\cM|_{E(H)}$  is a connected  matroid.
\end{remark}

\subsection{Count matroids}
\label{subsec:countM_intro}

Count matroids are a family of matroids defined on the edge set of a multigraph.
They occur frequently in combinatorial optimisation; see \cite{Fra}. They will be used as a source of examples throughout this paper.

Let $a,b$ be integers with $0\leq b<2a$ and $H = (V,E)$ be a multigraph. 
    The \emph{count matroid} $\cM_{a,b}(H)$ of $H$ is the matroid on $E$ in which a set of edges $F \subseteq E$ is independent if and only if $|F'| \leq a|V(F')| - b$ for all $\emptyset \neq F' \subseteq F$. Three well known examples of count matroids are $\cM_{1,1}(H)$, the \emph{cycle matroid of $H$}, $\cM_{1,0}(H)$, the \emph{bicycle matroid of $H$}, and $\cM_{2,3}(H)$, the \emph{$2$-dimensional rigidity matroid of $H$}.
    
    Note that, if $H$ and $H'$ are graphs and $F \subseteq E(H)\cap E(H')$  then the rank of $F$ will be the same in  $\cM_{a,b}(H)$ and $\cM_{a,b}(H')$.
    As such, we will generally not specify the underlying graph $H$ and assume it is a sufficiently dense
    complete multigraph that contains $F$.
    More precisely, we define the \emph{full count matroid} $\cM(a,b)$ to be the count matroid $\cM_{a,b}(H)$ where $H=K_n^{a,b}$  is a complete multigraph on $n$ vertices in which each vertex is incident with $\max\{a-b,0\}$ loops and all pairs of vertices are joined by $2a-b$ parallel edges. These parameters are chosen to ensure that $\cM(a,b)$ does not contain the trivial circuits consisting of a vertex with $a-b+1$ loops or a pair of vertices joined by $2a-b+1$ parallel edges.
    We will refer to such a complete multigraph as an \emph{ $(a,b)$-clique}.
    We will denote the rank function of $\cM(a,b)$ by $r_{a,b}$, and say that a graph $G\subseteq H$ is \emph{$\cM(a,b)$-rigid} if $r_{a,b}(G) = a|V(G)| - b$.
    Note that the $(a,b)$-clique on the same vertex set as $G$ also has rank $a|V(G)| - b$, so the $\cM(a,b)$-rigid graphs on a given vertex set have maximal possible rank.
    In addition, all induced subgraphs of $K_n^{a,b}$ 
    with at least two vertices 
    are $\cM(a,b)$-rigid. This statement remains true for induced subgraphs with one vertex when $0\leq b\leq a$.
    
We will need the following result which describes some elementary, and well known, properties of count matroids. 
In the lemma, and elsewhere in this paper, we take the \emph{degree} of a vertex $v$ in a multigraph to be the number of edges incident to $v$, counting each loop once only.

\begin{lemma} \label{lem:abclique} Let 
 $a,b,n$ be integers with
$0\leq b\leq 2a-1$ and $G=(V,E)\subseteq K_n^{a,b}$.
\begin{enumerate}
\item[(a)] Suppose $G-v$ is $\cM(a,b)$-independent for some vertex $v\in V$ of degree at most $a$. Then $G$ is  $\cM(a,b)$-independent.
\item[(b)] Suppose $G$ is an $\cM(a,b)$-circuit. Then $G$ 
is an $\cM(a,b)$-rigid connected graph of 
minimum degree at least $a+1$.
\item[(c)] Suppose $G$ is $\cM(a,b)$-rigid. Then the closure of $G$  in $\cM(a,b)$ is an $(a,b)$-clique on $V$.   
\end{enumerate}
\end{lemma}

\begin{proof}
(a) If $G=(V,E)$ is $\cM(a,b)$-dependent then $G$ will contain a subgraph $G'=(V',E')$ with $|E'|>a|V'|-b$. Then either $G'$, when $v\not\in V'$, or $G'-v$, when $v\in V'$,
contradicts the hypothesis that $G-v$ is  $\cM(a,b)$-independent.

(b) Since $G$ is an $\cM(a,b)$-circuit, $G-e$ is $\cM(a,b)$-independent and hence $|E-e|\leq a|V|-b$ for all $e \in E$.
    This gives $|E|=a|V|-b+1$, otherwise $G$ would be $\cM(a,b)$-independent. Hence $r_{a,b}(G) = |E|-1=a|V| - b$ and $G$ is $\cM(a,b)$-rigid. 
    
    To see that $G$ is connected, we suppose for a contradiction that $G=G_1\cup G_2$ where $G_i=(V_i,E_i)$ and $V_i\cap V_2=\emptyset$. Then $|E_i|\leq a|V_i|-b$ for both $i=1,2$ and hence $|E|=|E_1|+|E_2|\leq a|V|-2b$. This contradicts the fact that $|E|=a|V|-b+1$ by the previous paragraph.

The fact that $G$ has minimum degree at least $a+1$ follows immediately from (a).

(c) Let $K$ denote the edge set of the $(a,b)$-clique on $V$. For each $e\in K\setminus E$, the hypothesis that $G$ is $\cM(a,b)$-rigid implies that $r_{a,b}(G \cup e)=r_{a,b}(G)$ and hence $e$ is contained in the $\cM(a,b)$-closure of $G$. In addition, if $f$ is an edge of $\cM(a,b)$ whose end-vertices do not lie in $V$, then $f$ does not belong to the $\cM(a,b)$-closure of $G$ since every $\cM(a,b)$-circuit has minimum degree at least two by (b).  
\end{proof}

\subsection{Double circuits}
\label{subsec:doublec}

Lov\'asz \cite{Lov} introduced the notion of a double circuit in his solution to the matroid matching problem. We will review double circuits before extending this concept to that of a $k$-fold circuit in Section~\ref{sec:higher}.
Many of the properties of double circuits will be stated without proof: these can either be found in \cite{Lov, DL} or will be generalised to $k$-fold circuits in Section~\ref{sec:higher}.

Recall that a \emph{cyclic set} of a matroid $\cM=(E,r)$ is a subset $D \subseteq E$ which is the union of circuits of $\cM$ (equivalently  $r(D-e)=r(D)$ for all $e\in D$).
Thus $D$ is a  {double circuit} if $D$ is a cyclic set  with $r(D) = |D| -2$.
Let $\{A_1,A_2,\ldots,A_\ell\}$ be the principal partition of $D$.
Then $C \subset D$ is a circuit of $\cM$ if and only if $C=D-A_i$ for some $1\leq i\leq \ell$. A double circuit $D$ is said to be \emph{trivial} if $\ell=2$ and \emph{non-trivial} otherwise.
Note that $\cM|_ D$ is connected if and only $D$  is non-trivial.

As noted in the introduction,  the double circuits in the cycle matroid $\cM(1,1)$ are the subsets of the edge set of $K_n$ which induce pairs of edge-disjoint cycles or theta graphs. In the first case, the principal partition has exactly two parts and hence is trivial. In the latter, the principal partition is given by the partition of the theta graph into three internally disjoint paths and hence $\ell =3$. We next illustrate that the double circuits of the count matroid $\cM(2,3)$ have a richer structure. 

\begin{example}\label{ex:dc}
Figure \ref{fig:2double} provides four examples of double circuits in $\cM(2,3)$.
The double circuits in (a) and (b) are trivial. They both have exactly two $\cM(2,3)$-circuits, each isomorphic to $K_4$.
Their corresponding principal partitions into two copies of $K_4$ are immediate.
The double circuit $D$ in (c) has exactly three $\cM(2,3)$-circuits, two copies of $K_4$ and the graph obtained by deleting the edge $v_1v_2$ from $D$.
Hence the principal partition has three parts given by the complements of these three $\cM(2,3)$-circuits.
In (d), there are seven distinct $\cM(2,3)$-circuits, given by the unique copy of $K_4$ and the six (spanning) $\cM(2,3)$-circuits obtained by deleting exactly one edge from this $K_4$. Hence the principal partition is $\{A_1,A_2,\dots,A_7\}$ where $A_1$ is equal to the set of five edges incident to $u_1$ and $u_2$ and $A_2, A_3,\dots,A_7$ are singleton sets corresponding to the edges in the $K_4$.
\end{example}

\begin{figure}[ht]
\begin{center}
\begin{tikzpicture}[scale=0.6]

\filldraw (0,0) circle (3pt)node[anchor=north]{};
\filldraw (2,0) circle (3pt)node[anchor=north]{};
\filldraw (0,2) circle (3pt)node[anchor=north]{};
\filldraw (2,2) circle (3pt)node[anchor=north]{};

\filldraw (3,0) circle (3pt)node[anchor=north]{};
\filldraw (5,0) circle (3pt)node[anchor=north]{};
\filldraw (3,2) circle (3pt)node[anchor=north]{};
\filldraw (5,2) circle (3pt)node[anchor=north]{};

\draw[black,thick]
(0,0) --(2,0) -- (0,2) -- (2,2) -- (0,0) -- (0,2);

\draw[black,thick]
(2,0) -- (2,2);

\draw[black,thick]
(3,0) --(5,0) -- (3,2) -- (5,2) -- (3,0) -- (3,2);

\draw[black,thick]
(5,0) -- (5,2);

\filldraw (2.5,-1) circle (0pt)node[anchor=north]{(a)};


\filldraw (9.5,0) circle (3pt)node[anchor=north]{$w$};
\filldraw (7.5,0) circle (3pt)node[anchor=north]{};
\filldraw (11.5,0) circle (3pt)node[anchor=north]{};
\filldraw (7,2) circle (3pt)node[anchor=north]{};
\filldraw (9,2) circle (3pt)node[anchor=north]{};
\filldraw (10,2) circle (3pt)node[anchor=north]{};
\filldraw (12,2) circle (3pt)node[anchor=north]{};

\draw[black,thick]
(7,2) -- (9.5,0) -- (7.5,0) -- (7,2) -- (9,2) -- (9.5,0) -- (12,2) -- (11.5,0) -- (9.5,0) -- (10,2) -- (12,2);

\draw[black,thick]
(7.5,0) -- (9,2);

\draw[black,thick]
(10,2) -- (11.5,0);

\filldraw (9.5,-1) circle (0pt)node[anchor=north]{(b)};

\filldraw (16.5,0) circle (3pt)node[anchor=north]{$v_1$};
\filldraw (16.5,2) circle (3pt)node[anchor=south]{$v_2$};
\filldraw (14.5,0) circle (3pt)node[anchor=north]{};
\filldraw (14.5,2) circle (3pt)node[anchor=north]{};
\filldraw (18.5,0) circle (3pt)node[anchor=north]{};
\filldraw (18.5,2) circle (3pt)node[anchor=north]{};

\draw[black,thick]
(14.5,0) -- (16.5,0) -- (18.5,0) -- (18.5,2) -- (16.5,2) -- (14.5,2) -- (14.5,0) -- (16.5,2) -- (18.5,0);

\draw[black,thick]
(14.5,2) -- (16.5,0) -- (18.5,2);

\draw[black,thick]
(16.5,0) -- (16.5,2);

\filldraw (16.5,-1) circle (0pt)node[anchor=north]{(c)};

\filldraw (22.5,0) circle (3pt)node[anchor=north]{};
\filldraw (22.5,2) circle (3pt)node[anchor=north]{};
\filldraw (24.5,0) circle (3pt)node[anchor=north]{};
\filldraw (24.5,2) circle (3pt)node[anchor=north]{};
\filldraw (21.5,1) circle (3pt)node[anchor=east]{$u_1$};
\filldraw (25.5,1) circle (3pt)node[anchor=west]{$u_2$};

\draw[black,thick]
(22.5,0) -- (22.5,2) -- (24.5,0) -- (24.5,2) -- (22.5,0) -- (21.5,1) -- (22.5,2) -- (24.5,2) -- (25.5,1) -- (24.5,0) -- (22.5,0);

\draw[black,thick]
(21.5,1) -- (25.5,1);

\filldraw (23.5,-1) circle (0pt)node[anchor=north]{(d)};
\end{tikzpicture}
\end{center}
\caption{Four examples of double circuits in $\cM(2,3)$.}
\label{fig:2double}
\end{figure}
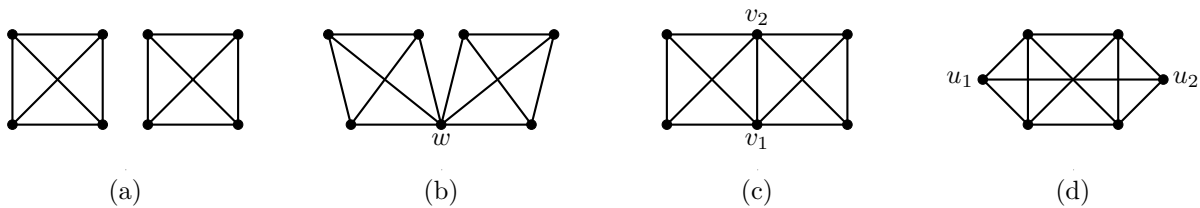

We will see that  a set $D$ is a non-trivial double circuit in a matroid $\cM$ if and only if $\cM|_D$ has an ear decomposition with precisely two ears.
We will formally prove this statement along with a more general statement for $k$-fold circuits in Proposition~\ref{prop:kcircuitears}.

\section{\boldmath $k$-fold circuits}
\label{sec:higher}

Double circuits can be naturally generalised to triple circuits and beyond: 

\begin{definition}
    Let $\cM = (E,r)$ be a matroid 
    and $k \in \NN$ a nonnegative integer.
    A \emph{$k$-fold circuit}   of $\cM$ is a cyclic set $D \subseteq E$ with $r(D) = |D| - k$.
\end{definition}

Note that the empty set is the unique 0-fold circuit in any matroid $\cM$.

\begin{example} If $\cM$ is the cycle matroid of a complete graph  $K_n$ then the $k$-fold circuits of $\cM$ are the edge sets of the bridgeless subgraphs $G=(V,D)$ with   $c$ connected components which satisfy $k=|D|-|V|+c$. 
This holds since the rank of $D$ is $|V|-c$ (the size of a maximal forest in $G$) and the bridgeless condition ensures that every edge in $D$ is contained in at least one cycle of $G$.
\end{example}

There is a natural bijection between the $k$-fold circuits in a matroid and the flats of its dual matroid of corank $k$.
This is hinted at in the proof of \cite[Lemma 1.2]{DL}, but we make the relationship explicit in our next result.

\begin{lemma}\label{lem:k-circuit+coflat}
    Let $\cM = (E,r)$ be a matroid whose dual matroid $\cM^*$ has rank $t$. Then $D \subseteq E$ is a $k$-fold circuit of $\cM$ if and only if $E \setminus D$ is a flat of rank $t - k$ of $\cM^*$.
\end{lemma}

\begin{proof}
    It is well known that $C$ is a circuit in $\cM$ if and only if $E \setminus C$ is a hyperplane of $\cM^*$~\cite[Proposition 2.1.6]{Oxl}.
    It immediate follows that $D$ is cyclic  (i.e., a union of circuits) in $\cM$ if and only if $E \setminus D$ is a flat (i.e., an intersection of hyperplanes) of $\cM^*$.
    Writing $r(D)$ in terms of $r^*$ gives the equality $r(D) = |D| - k$ precisely when $r^*(E \setminus D) = t-k$.
\end{proof}

Let $\cD(\cM)$ be the set of all cyclic sets in a matroid $\cM$.
It was shown by Tutte~\cite[Section 4]{Tutte} that $\cD(\cM)$ has a lattice structure with partial order $D < D'$ given by inclusion.
It follows from Lemma~\ref{lem:k-circuit+coflat} that there is an order-reversing bijection between $\cD(\cM)$ and the lattice of flats $\cL(\cM^*)$ of $\cM^*$
\[
\cD(\cM) \rightarrow \cL(\cM^*) \, , \quad D \mapsto E \setminus D \, .
\]
The following properties are routine to check from this observation.

\begin{lemma}\label{lem:circuit+lattice}
    $\cD(\cM)$ is a graded atomic lattice.
    Its join and meet operations are
    \[
    D \vee D' = D \cup D' \, , \quad D \wedge D' = \{e \in D \cap D' \mid r(D \cap D') = r((D \cap D') - e) \} \, .
    \]
    Its grading is $\rho \colon \cD(\cM) \rightarrow \ZZ_{\geq 0}$ where $\rho(D) = |D| - r(D)$, i.e., $\rho(D) = k$ if and only if $D$ is a $k$-fold circuit.
\end{lemma}

We note that the empty set, the unique 0-fold circuit, is the minimum element of this lattice.
The grading on $\cD(\cM)$ is known as the \emph{nullity} elsewhere in the literature.

\begin{remark}
    Though the lattice of flats $\cL(\cM^*)$ is semimodular,
    this property is not passed on to $\cD(\cM)$, as the bijection is order-reversing and hence reverses the semimodular inequality, i.e.,
    \[
    \rho(D) + \rho(D') \leq \rho(D\vee D') + \rho(D \wedge D') \, .
    \]
\end{remark}

We now extend the notion of a principal partition to $k$-fold circuits.
Let $D$ be a $k$-fold circuit of $\cM =(E,r)$, each of the $(k-1)$-fold circuits of $\cM$ contained in $D$ are of the form $D \setminus A_i$ for some $A_i \subset D$.
Moreover, these are exactly the sets covered by $D$ in $\cD(\cM)$: this is depicted in Figure~\ref{fig:lattice+dual}.
Applying Lemmas~\ref{lem:k-circuit+coflat} and \ref{lem:circuit+lattice}, we have that $E \setminus D$ is a flat of $\cM^*$ that is covered by the flats $\{(E \setminus D) \cup A_1, \dots, (E \setminus D) \cup A_\ell\}$.
Applying the covering axiom (Lemma~\ref{lem:covering}) to $E\setminus D$ in $\cL(\cM^*)$, it follows that the sets $\{A_1, \dots, A_\ell\}$ form a partition of $D$.
This leads to the following generalisation of the principal partition from double circuits to $k$-fold circuits.

\begin{figure}
    \centering
    \includegraphics[width=\linewidth]{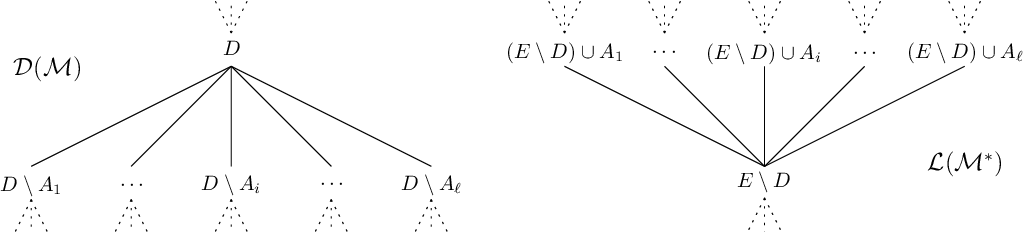}
    \caption{The lattice $\cD(\cM)$ of cyclic sets of $\cM$ and the lattice $\cL(\cM^*)$ of flats of $\cM^*$.
    The two lattices are anti-isomorphic via the map $D \mapsto E\setminus D$.
    Moreover, the covering axiom for flats in $\cL(\cM^*)$ implies the notion of the principal partition for $k$-fold circuits.}
    \label{fig:lattice+dual}
\end{figure}

\begin{definition}\label{def:principalpartfork}
     Let $\cM=(E,r)$ be a matroid and $D \subseteq E$ be a $k$-fold circuit of $\cM$. 
     The \emph{principal partition of D} is the partition $\{A_1, \dots, A_\ell\}$ of $D$ where $\{D \setminus A_i \mid 1 \leq i \leq \ell\}$ is the set of all $(k-1)$-fold circuits of $\cM$ contained in $D$.
\end{definition}

\begin{remark}\label{rem:lattice+restriction}
    If we are analysing some fixed $k$-fold circuit $D$ in $\cM$, we may as well restrict to the matroid $\cM|_D$ as elements of the ground set outside of $D$ do not impact its $k$-fold circuit structure.
    This simplifies the lattice of flats of its dual also: it is straightforward to check that $\cL((M|_D)^*)$ is isomorphic to the interval $[E \setminus D, E]$ in $\cL(\cM^*)$ via the map $X \mapsto X \cup (E \setminus D)$.
    In particular, it is a rank $k$ lattice whose atoms are precisely the parts of the principal partition.
    We will frequently assume that $\cM = \cM|_D$ and utilise this observation throughout proofs in this section.
\end{remark}

The following proposition states that we can equivalently define the principal partition of a $k$-fold circuit in terms of an equivalence relation, and that there is a lower bound on the number of parts.

\begin{proposition} \label{prop:principal+partition+k}
    Let $\cM=(E,r)$ be a matroid and $D \subseteq E$ be a $k$-fold circuit of $\cM$.
    \begin{enumerate}[label=(\alph*)]
        \item The principal partition of $D$ has at least $k$ parts,
        \item Two elements $e,f \in D$ are contained in the same part of the principal partition if and only if $r(D-e-f) = r(D) -1$.
    \end{enumerate}
\end{proposition}

\begin{proof}
    Without loss of generality, let $\cM = \cM|_D$ and $\cM^* = (\cM|_D)^*$ with rank function $r^*$.
    By Remark~\ref{rem:lattice+restriction}, $\cL(\cM^*)$ is a rank $k$ lattice whose atoms are the parts of the principal partition.
    Part (a) follows from the observation that a rank $k$ lattice must have at least $k$ atoms.
    For part (b), the dual rank formula gives
    \[
    r(D - e - f) = r(D) - |{e,f}| + r^*(e,f) \, .
    \]
    This implies $r(D - e -f) = r(D) -1$ if and only if $r^*(e,f) = 1$, or equivalently $e,f$ are contained in the same rank one flat of $\cM^*$.
    As these are the parts of the principal partition, the claim follows.
\end{proof}

The following lemma restricts how parts of the principal partition can interact with circuits of the matroid, which will be useful throughout.

\begin{lemma} \label{lem:PP+circuit+interaction}
    Let $D$ be a $k$-fold circuit in a matroid $\cM$ and $A_i$ a part of the principal partition of $D$.
    For any circuit $C \subseteq D$, either $C \cap A_i = \emptyset$ or $A_i \subseteq C$.
\end{lemma}

\begin{proof}
Suppose $C \cap A_i \neq \emptyset$ and take $a \in C \cap A_i$.
    For any $b \in A_i - a$, we have $r(D - b - a) = r(D - b) - 1$ by Proposition~\ref{prop:principal+partition+k} (b).
    As $a$ is a coloop in $D-b$ but contained in the circuit $C \subseteq D$, this can only occur if $C \nsubseteq D -b$. Hence $b \in C$ and $A_i \subseteq C$.
\end{proof}

We call a $k$-fold circuit $D$ with principal partition $\{A_1, \dots, A_\ell\}$ \emph{trivial} if $\ell=k$, and non-trivial if $\ell > k$. Our next result  determines the structure of trivial $k$-fold circuits.

\begin{lemma}\label{lem:trivialkcircuit}
    Let $D$ be a trivial $k$-fold circuit in a matroid $\cM$ and $\{A_1,A_2,\ldots,A_k\}$ be its principal partition. Then $A_i$ is a circuit of $\cM$ for all $1\leq i\leq k$ and $\cM|_D=\cM|_{A_1}\oplus \cM|_{A_2}\oplus \cdots \oplus \cM|_{A_k}$.
    \end{lemma}

\begin{proof}
    Again, without loss of generality let $\cM = \cM|_D = (D,r)$ and $\cM^* = *\cM|_D)^* = (D, r^*)$ be its matroid dual. Then $r^*(D)=|D|-r(D)=k$. Since each $A_i$ is a rank 1 flat of $\cM^*$ by Lemma \ref{lem:k-circuit+coflat}, we have $r^*(D)=k=r^*(A_1)+r^*(A_2)+\cdots+r^*(A_k)$.
    It follows from~\cite[Proposition 4.2.1]{Oxl} that $\cM^*=\cM^*|_{A_1}\oplus \cM^*|_{A_2}\oplus \cdots \oplus \cM^*|_{A_k}$. 
    This implies that $D\setminus A_i$ is a rank $(k-1)$ flat in $\cM^*$. We can now apply Lemma \ref{lem:k-circuit+coflat} again to deduce that each $A_i$ is a circuit in $\cM$ and $\cM=\cM|_{A_1}\oplus \cM|_{A_2}\oplus \cdots \oplus \cM|_{A_k}$.
\end{proof}

Lemma \ref{lem:trivialkcircuit} implies that the trivial $k$-fold circuits of a matroid $\cM$ are not $\cM$-connected whenever $k\geq 2$. The converse is true for $k=2$ by Proposition \ref{prop:decompose} below, but the following example shows it can be false when $k\geq 3$.

\begin{example}
  We give an example of two non-trivial triple circuits with the same principal partition, one of which is $\cM$-connected and one of which is not.

    Let $\cM= U_{1,4}$ be the rank one uniform matroid on four elements and let $T = \{1,2,3,4\}$ be its unique triple circuit.
    Any 3-subset is a double circuit, and so the principal partition of $T$ has four parts.
    Moreover, any pair of elements of $T$ is a circuit of $\cM$, and so it is $\cM$-connected.

    Now let $\cM' = U_{1,3} \oplus U_{0,0}$, obtained from $\cM$ by making the singleton set $\{4\}$ dependent.
    Again, $T = \{1,2,3,4\}$ is its unique triple circuit.
    Moreover, it is straightforward to check that any 3-subset gives a double circuit of $\cM'$, and so the principal partition of $T$ has four parts.
    However, $T$ is clearly not $\cM'$-connected as the only circuit of $\cM'$ containing $\{4\}$ is itself.
    For triple circuits and $k$-fold circuits more generally, connectivity cannot be read off from the principal partition as it only gives direct information about the $(k-1)$-fold circuits.
\end{example}

We can characterise connected $k$-fold circuits by the existence of ear decompositions into $k$ ears.

\begin{proposition}\label{prop:kcircuitears}
    Let $\cM = (E,r)$ be a matroid and $D \subseteq E$.
    Then $D$ is an $\cM$-connected $k$-fold circuit if and only if $\cM|_D$ has an ear decomposition into $k$ ears.
\end{proposition}

\begin{proof}
Lemma~\ref{lem:ear}(1) gives that $D$ is $\cM$-connected if and only if $\cM|_D$ has an ear decomposition.
Let $(C_1, \dots, C_m)$ be such an ear decomposition, with $D_i = C_1 \cup \cdots \cup C_i$.
We show by induction that $r(D_i) = |D_i| - i$ for all $1 \leq i \leq m$.
The base case follows from $D_1 = C_1$ being a circuit.
For the inductive step, Lemma~\ref{lem:ear}(3) gives us
\[
r(D_i) = r(D_i) + |D_i \setminus D_{i-1}| -1 = |D_{i-1}| - (i-1) + |D_i \setminus D_{i-1}| - 1 = |D_i| - i \, .
\]
It follows that $D$ is a $k$-fold circuit if and only if the ear decomposition has $k$ ears.
\end{proof}

We end this section by demonstrating how $k$-fold circuits behave when combined under various operations.

\begin{proposition}\label{prop:decompose}
    Let $\cM = (E,r)$ be a matroid and $D,D_1,\ldots,D_s \subseteq E$ such that  $\cM|_D = \bigoplus_{i=1}^s\cM|_{D_i}$.
    Then $D$ is a $k$-fold circuit if and only if each $D_i$ is a $k_i$-fold circuit and $k = \sum_{i=1}^s k_i$.
    Moreover, if $D_i$ has principal partition $\cA_i$, then the principal partition of $D$ is $\cA = \bigcup_{i=1}^s \cA_i$.
\end{proposition}

\begin{proof}
    Again without loss of generality, we let $\cM = \cM|_D$, i.e., $D = E$.
    Let $r_i = r|_{D_i}$, then for all $X\subseteq D$, we have
    $r(X) = \sum_{i=1}^s r_i(X \cap D_i)$ (see \cite[4.2.13]{Oxl}). 
    This immediately implies that $D$ is a cyclic set in $\cM$ if and only if each $D_i$ is a cyclic set in $\cM$. 
    In addition, if either statement holds then we have 
$ |D| - k = r(D) = \sum_{i=1}^s r_i(D_i) = \sum_{i=1}^s |D_i| - k_i  $
    and hence $k = \sum_{i=1}^s k_i$.

    By Proposition~\ref{prop:principal+partition+k}, $e,f \in D$ are in the same part of $\cA$ if and only if
    \[
    \sum_{i=1}^s r_i(D_i -e -f) = r(D- e - f) = r(D) - 1 = \sum_{i=1}^s r_i(D_i) -1 \, .
    \]
    To decrease the rank of $\sum_{i=1}^s r_i(D_i)$, both $e$ and $f$ must be removed from the same $D_i$.
    The rank then decreases if and only if $e,f$ are in the same part of $\cA_i$.
    Hence the principal partition of $D$ is just the union of the principal partitions of each $D_i$.
\end{proof}

\section{\boldmath The $k$-fold circuit property}
\label{sec:hcp}

Recall that a matroid $\cM$ has the \emph{double circuit property} if every double circuit $D$ in $\cM$ satisfies \eqref{eq:dcpalt}.
Dress and Lov\'{a}sz in \cite{DL} showed that many natural classes of matroids either satisfy this 
property, or can be embedded inside a matroid that satisfies it (see Example~\ref{ex:full+matroids} below).

Given a  $k$-fold circuit $D$ in a matroid $\cM$ with principal partition $\{A_1,A_2,\dots,A_\ell\}$, we can consider when the inequality $r(\bigcap_{i=1}^\ell \cl(D\setminus A_i)) \geq \ell-k$ holds. This is a straightforward generalisation of \eqref{eq:dcpalt} to $k$-fold circuits. We first 
show that the reverse inequality holds in all matroids.  This implies, in particular, that \eqref{eq:dcpalt} will hold \emph{with equality} in all matroids which have the double circuit property.

\begin{theorem}\label{thm:strong}
    Let $D$ be  a $k$-fold circuit in a matroid $\cM=(E,r)$ and $\{A_1, \dots, A_\ell\}$ be the principal partition of $D$. Then 
    \begin{equation}\label{eq:dcpkrev}
   r\left(\bigcap_{i=1}^\ell \cl(D\setminus A_i)\right) \leq \ell-k\,.
\end{equation} 
In addition, \eqref{eq:dcpkrev} holds with equality {if and only if} $\bigcap_{i=1}^{n-1} \cl(D\setminus A_i)$ and $\cl(D\setminus A_n)$ is a modular pair of flats of $\cM$ for all $2\leq n\leq \ell$.
\end{theorem}

\begin{proof}
    We  use  
    induction on $m$ to 
    show that, for all $1\leq m\leq \ell$, 
    \begin{equation} \label{eq:Wi+dimensionrev}
    r\left(\bigcap_{i=1}^m \cl(D\setminus A_i)\right) \leq r(D) - \left(\sum_{j=1}^m |A_j|\right) + m \, ,
    \end{equation}
    with equality if and only if $\bigcap_{i=1}^{n-1} \cl(D\setminus A_i)$ and $\cl(D\setminus A_n)$ is a modular pair of flats of $\cM$ for all $2\leq n\leq m$.
    We first consider the base case $m=1$. By the definition of the principal partition, removing one element of $A_1$ from $D$ does not decrease the rank, but removing every subsequent element decreases the rank by one.
    This immediately gives  $r(\cl(D \setminus A_1))=r(D \setminus A_1) = r(D) - |A_1| + 1$, so (\ref{eq:Wi+dimensionrev}) holds with equality when $m=1$.

    Now consider $r(\bigcap_{i=1}^m \cl(D\setminus A_i))$ when $m\geq 2$.
    By (sub)modularity and induction, 
    we have
    \begin{align} \label{eq:linear+int+dimensionrev}
    r\left(\bigcap_{i=1}^m \cl(D\setminus A_i)\right) &\leq  r\left(\bigcap_{i=1}^{m-1} \cl(D\setminus A_i)\right) + r\left(\cl(D\setminus A_m)\right) - r\left(\left(\bigcap_{i=1}^{m-1} \cl(D\setminus A_i)\right) \bigcup \cl(D\setminus A_m)\right) \nonumber \\
    &\leq  2 \cdot r(D) - \left(\sum_{i=1}^m |A_j|\right) + m - r\left(\left(\bigcap_{i=1}^{m-1}\cl(D\setminus A_i)\right) \bigcup \cl(D\setminus A_m)\right) \, ,
    \end{align}
    and equality holds throughout if and only if  $\bigcap_{i=1}^{n-1} \cl(D\setminus A_i)$ and $\cl(D\setminus A_n)$ is a modular pair of flats of $\cM$ for all $2\leq n\leq m$.
    Since $D\subseteq (\bigcap_{i=1}^{m-1}\cl(D\setminus A_i)) \cup \cl(D\setminus A_m)\subseteq \cl(D)$, we have
    \[
     r\left(\left(\bigcap_{i=1}^{m-1}\cl(D\setminus A_i)\right) \cup \cl(D\setminus A_m)\right)= r(D)\, .
    \]
    Substituting this equation into \eqref{eq:linear+int+dimensionrev} gives  \eqref{eq:Wi+dimensionrev}, with equality if and only if $\bigcap_{i=1}^{n-1} \cl(D\setminus A_i)$ and $\cl(D\setminus A_n)$ is a modular pair of flats of $\cM$ for all $2\leq n\leq m$.

    Finally, we can combine \eqref{eq:Wi+dimensionrev} in the case where $m = \ell$ with the fact that $r(D) = |D| - k$, to obtain
    \[
    r\left(\bigcap_{i=1}^{\ell} \cl(D\setminus A_i)\right) \leq  r(D) - \left(\sum_{j=1}^\ell |A_j|\right) + \ell = |D| - k - |D| + \ell = \ell - k \, ,
    \]
    with equality if and only if $\bigcap_{i=1}^{n-1} \cl(D\setminus A_i)$ and $\cl(D\setminus A_n)$ is a modular pair of flats of $\cM$ for all $2\leq n\leq \ell$.
    This completes the proof of the theorem.
    \end{proof} 

Theorem \ref{thm:strong} suggests the following extension of the double circuit property 
to $k$-fold circuits. 

\begin{definition}
A $k$-fold circuit $D$ with principal partition $\{A_1,A_2,\dots,A_\ell\}$ in a matroid $\cM=(E,r)$ is \emph{balanced} if it satisfies
\begin{equation}\label{eq:dcpkalt}
   r(\bigcap_{i=1}^\ell \cl(D\setminus A_i)) = \ell-k.
\end{equation} 
The matroid $\cM$ has the \emph{$k$-fold circuit property} if all of its $k$-fold circuits are balanced.
\end{definition}
We will use the convention that if a matroid has no $k$-fold circuits, it trivially satisfies the $k$-fold circuit property.

The 2-fold circuits of $\cM(2,3)$ given in Example \ref{ex:dc} are all balanced. In (a) and (b) we have 
$r_{2,3}(\bigcap_{i=1}^2\cl(D\setminus A_i))=r_{2,3}(\emptyset)=0$, in (c) we have 
$r_{2,3}(\bigcap_{i=1}^3\cl(D\setminus A_i))=r_{2,3}(K_2)=1$, and in (d) we have 
 $r_{2,3}(\bigcap_{i=1}^7\cl(D\setminus A_i))=r_{2,3}(K_4)=5$.

Note that every matroid $\cM$ has the 1-fold circuit property since, if $C$ is a circuit in $\cM$, then its principal partition is $\{C\}$ and we have $r(\cl(C\setminus C))=r(\emptyset)=0=1-1$. Theorem \ref{thm:strong} immediately gives us a class of matroids which have the $k$-fold circuit property for all $k\geq 1$.

\begin{corollary}\label{cor:cyclicflats} Let $\cM$ be a matroid and suppose that all pairs of flats of $\cM$ are modular. Then $\cM$ has the $k$-fold circuit property for all $k\geq 1$.
\end{corollary}

The second part of Theorem \ref{thm:strong} gives a characterisation of balanced $k$-fold circuits.

\begin{corollary}\label{cor:balanced+modular+sublattice}
A $k$-fold circuit $D$ of $\cM$ with principal partition $\{A_1, \dots, A_\ell\}$ is balanced if and only if $\bigcap_{i\in I} \cl(D\setminus A_i)$ and $\cl(D\setminus A_j)$ is a modular pair of flats of $\cM$ for all $\emptyset \subsetneq I \subsetneq [\ell]$ and $j \in [\ell] \setminus I$.   
\end{corollary}

\begin{proof}
    Theorem \ref{thm:strong} gives the claim for all sets of the form $I = \{1, \dots, n-1\}$ and $j = n$.
    As the definition of balanced is independent of the ordering of the principal partition, we can reorder the principal partition and apply Theorem \ref{thm:strong} to give the claim for any $\emptyset \subsetneq I \subsetneq [\ell]$ and $j \in [\ell] \setminus I$.
\end{proof}

Our characterisation shows that the $k$-fold circuit property is connected to modular structures within the matroid $\cM$.
We shall make this connection more precise 
in Section~\ref{sec:pseudo}.

We next focus on  matroids that do satisfy the $k$-fold circuit property.
Dress and Lov\'{a}sz showed that a number of `full' matroids satisfy this property. 
The \emph{full linear matroid} $L^s_\KK$ of rank $s$ over a field $\KK$ is 
the matroid on $\KK^s$ in which independence is given by linear independence in the vector space $\KK^s$. Dress and Lov\'asz showed that $L^s_\KK$ has the double circuit property. Their result can be extended to the $k$-fold circuit property by applying Corollary \ref{cor:cyclicflats}, using the fact that  the flats of $L^s_\KK$ are the subspaces of $\KK^s$ so all pairs of flats are modular.  

\begin{corollary}\label{cor:fulllinear}
    Full linear matroids have the $k$-fold circuit property for all $k\geq 1$.
\end{corollary}

    Every matroid $\cM$ of rank $s$ that is representable over $\KK$ can be embedded inside the full linear matroid $L_\KK^s$, where the embedding preserves (in)dependence.
    However, the closure of a set in $L_\KK^s$ is a subspace of $\KK^s$ and hence can be much larger than in $\cM$. 
    This allows an unbalanced $k$-fold circuit of $\cM$ to become balanced when considered as a $k$-fold  circuit in $L^s_\KK$.

\begin{example}\label{ex:full+matroids}
    Let $\cM = (D,r)$ be the column matroid of the real matrix
    \[
    \begin{blockarray}{cccccc}
    v_1 & v_2 & v_3 & v_4 & v_5 & v_6 \\
    \begin{block}{[cccccc]}
        1 & 0 & 0 & 0 & 1 & 1 \\
        0 & 1 & 0 & 0 & -1 & 0 \\
        0 & 0 & 1 & 0 & 0 & -1 \\
        0 & 0 & 0 & 1 & 1 & 1 \\
        \end{block}
    \end{blockarray} \, ,
    \]
    where $D = \{v_1, \dots, v_6\}$.
    It is straightforward to check that $D$ is a double circuit with principal partition $\{\{v_1,v_4\},\{v_2,v_5\},\{v_3,v_6\}\}$.
    As such, $D$ is not balanced in $\cM$, as
    \[
    r\left(\bigcap_{i=1}^3 \cl_{\cM}(D \setminus A_i)\right) = r(\{v_1, v_2,v_4,v_5\} \cap \{v_1, v_3,v_4,v_6\} \cap \{v_2,v_3,v_5,v_6\}) = r(\emptyset) < 3-2 \, .
    \]
    However, embedding $\cM$ inside the full linear matroid $L_{\RR}^4$, we see that $D$ is balanced in $L_{\RR}^4$:
    \begin{align*}
    r\left(\bigcap_{i=1}^3 \cl_{L_{\RR}^4}(D \setminus A_i)\right) &= r({\rm span}\{v_1, v_2,v_4,v_5\} \cap {\rm span}\{v_1, v_3,v_4,v_6\} \cap {\rm span}\{v_2,v_3,v_5,v_6\}) \\
    &= r({\rm span}(1,0,0,1)) = 1=3 - 2  \, .
    \end{align*}
    \end{example}

    Lov\'asz and Dress showed that other families of `full' matroids
    satisfy the double circuit property in~\cite{DL}. 
    We will extend these results to the $k$-fold circuit property in Section \ref{sec:pseudo}.

We next give a large family of matroids which have the $k$-fold circuit property for all $k\geq 3$.
A matroid $\cM$ of rank $s$ is \emph{paving} if every circuit $C \in \cC(\cM)$ has $|C| = s$ or $s+1$.
A paving matroid is \emph{sparse} if every circuit $C$ of cardinality $s$ is a flat, i.e. $r(C + e) > r(C) = s-1$.

\begin{proposition}\label{prop:sparse+paving}
    Let $\cM = (E,r)$ be a sparse paving matroid.
    Then $\cM$ has the $k$-fold circuit property for all $k \geq 3$.
\end{proposition}
    
\begin{proof}
    We first show that for all $2 \leq k \leq |E| - s$, every $D \in \binom{E}{s+k}$ is a $k$-fold circuit of rank $s$, i.e. $r(D) = r(D-e) = s$ for all $e \in D$.
    We proceed by induction on $k$.
    Let $D \in \binom{E}{s+2}$ and $e \in D$, we show $r(D) = r(D-e) = s$ and hence $D$ a double circuit.
    If $D - e$ a circuit then this follows immediately from $|D -e| = s+1$.
    If $D-e$ is not a circuit, then it properly contains some circuit $C$.
    As $\cM$ is sparse paving, we have $|C| = s$ and $s-1 = r(C) < r(D-e) = r(D) = s$.
    
    For the general case, if $D \in \binom{E}{s+k}$ then $s = r(D-e) = r(D)$ by the inductive hypothesis, utilising that $D-e$ is a $(k-1)$-fold circuit. Let $D$ be a $k$-fold circuit of $\cM$ for $k \geq 3$.
    The earlier claim tells us that for all $e \in D$, we have $D-e$ is a $(k-1)$-fold circuit, and $\cl(D-e) = E$.
    It follows that $D$ is a balanced $k$-fold circuit, as
    \[
    r\left(\bigcap_{e \in D} \cl(D -e) \right) = r(E) = s = |D| - k \, .
    \]
\end{proof}

We can use Proposition \ref{prop:sparse+paving} to
construct examples of matroids that do not satisfy the double circuit property, but do satisfy the $k$-fold circuit property for all $k \geq 3$.

\begin{example}\label{ex:sparse+paving}
    Let $E = \{1, \dots, 2t + 4\}$ where $t\geq 2$ and consider $\cH = \{H_i \mid 1 \leq i \leq t+2\}$ where $H_i = \{2i-1, 2i\}$.
    Let $\cM$ be the rank $2t$ matroid on $E$ whose circuits are
    \[
    \cC(\cM) = \{E \setminus (H_i \cup H_j) \mid 1 \leq i \neq j \leq t+2\} \cup \{E \setminus \{a,b,c\} \mid H_i \nsubseteq \{a,b,c\} \text{ for all } 1 \leq i \leq t+2 \} \, .
    \]
    It is straightforward to verify that $\cM$ is a sparse paving matroid, hence it satisfies the $k$-fold circuit property for $k \geq 3$.
    However, it does not satisfy the double circuit property: let $D = \{1, \dots, 2t+2\} = E \setminus H_{t+2}$.
    It follows $D$ is a double circuit with principal partition $\{H_1, \dots, H_{s+1}\}$.
    By the definition of sparse paving, each $D \setminus H_i$ is a flat, and hence
    \[
    r\left(\bigcap_{i=1}^{t+1} \cl(D \setminus H_i)\right) = r(\emptyset) = 0 < (t+1) -2 \, .
    \]
\end{example}

\begin{remark}
    It is conjectured that almost all matroids are sparse paving matroids~\cite{MNWW}. If true, Proposition \ref{prop:sparse+paving} would imply  that almost all matroids satisfy the $k$-fold circuit property for all $k \geq 3$.
\end{remark}

We next show there exists matroids that fail the $k$-fold circuit property for arbitrarily high $k$.
Explicitly, for any $m \in \NN$, we can find some matroid that does not have the $k$-fold circuit property for all $2 \leq k \leq m$.
We let $P_{j}$ be the graph with vertices $u, v$ and $\{w_1, \dots, w_j\}$, and edges $\{uw_i \mid 1 \leq i \leq j\}$ and $\{vw_i \mid 1 \leq i \leq j\}$.
Observe that $P_j$ is the edge-disjoint union of $j$ paths of length two between $u$ and $v$.

\begin{proposition} \label{prop:kfold+failure}
    Let $G$ be a graph that contains $P_{m+1}$ as an induced subgraph.
    Then the cycle matroid $\cM_{1,1}(G)$ does not have the $k$-fold circuit property for all $2 \leq k \leq m$.
\end{proposition}

\begin{proof}
Write $\cM := \cM_{1,1}(G)$.
 Note that if $G$ contains $P_{m+1}$ as an induced subgraph, it also contains $P_{j}$ as an induced subgraph for all $j \leq m+1$.
 Hence it suffices to prove $\cM$ does not have the $m$-fold circuit property.

 It is straightforward to verify that $P_{m+1}$ is an $m$-fold circuit of $\cM$, and that its principal partition is $\{A_1, \dots, A_{m+1}\}$ where $A_i = \{uw_i, vw_i\}$.
 Moreover, $P_{m+1} \setminus A_i \cong P_m$ and is already closed in $\cM$.
 As there is no edge contained in all copies of $P_{m+1} \setminus A_i$, it follows that
 \[
 r\left(\bigcap_{i=1}^{m+1} \cl(P_{m+1} \setminus A_i)\right) = r(\emptyset) = 0 < m+1 - m \, .
 \]
 Hence $P_{m+1}$ is not balanced, and so $\cM$ does not satisfy the $m$-fold circuit property.
\end{proof}

We next consider the balanced property for a disconnected $k$-fold circuit in a matroid $\cM$. 
We will need the following elementary lemma.

\begin{lemma}\label{lem:elemalt} Let $\cM=(E,r)$ be a matroid. Suppose that  $r(X_1\cup X_2)=r(X_1)+ r(X_2)$ for some $X_1,X_2\subseteq E$. Then $r(\cl(X_1)\cup \cl(X_2))=r(\cl(X_1))+ r(\cl(X_2))$.
\end{lemma}

\begin{proof} 
We can use the submodularity and monotonicity of $r$ to deduce that
$$r(X_1)+r(X_2)=r(X_1\cup X_2)\leq r(\cl(X_1)\cup\cl(X_2))\leq r(\cl(X_1))+r(\cl(X_2))=r(X_1)+r(X_2).$$
Hence equality must hold throughout. In particular, we have $r(\cl(X_1)\cup \cl(X_2))=r(\cl(X_1))+ r(\cl(X_2))$.
\end{proof}

We saw in Proposition \ref{prop:decompose} that the direct sum of a  $k_1$-fold circuit and a $k_2$-fold circuit is a $(k_1+k_2)$-fold circuit. Our next result shows that the balanced property is preserved by this operation.

\begin{lemma}\label{lem:disconalt} Suppose  $D_i$ is a balanced $k_i$-fold circuit in a matroid $\cM$ for both  $i=1,2$ and $\cM|_{D_1\cup D_2}=\cM|_{D_1}\oplus \cM|_{D_2}$. Then $D_1\cup D_2$  is a balanced $(k_1+k_2)$-fold circuit in $\cM$.
\end{lemma}

\begin{proof} 
Let $D=D_1\cup D_2$, $k=k_1+k_2$ and let $\cA_i=\{A^i_1,A^i_2,\ldots,A^i_{\ell_i}\}$ be the principal partition of $D_i$ for $i=1, 2$. Then $D$  is a  $k$-fold circuit in $\cM$ and $\cA_1\cup \cA_2$ is the principal partition of $D$ by Proposition \ref{prop:decompose}.
Since each $D_i$ is balanced, we have $r(\bigcap_{j=1}^{\ell_i}\cl(D_i\setminus A^i_j))= \ell_i-k_i$ for both $i=1, 2$.  

Let  $F=\left(\bigcap_{j=1}^{\ell_1}\cl(D\setminus A^1_j)\right)\cap \left(\bigcap_{j=1}^{\ell_2}\cl(D\setminus A^2_j)\right)$, $F_1=\bigcap_{j=1}^{\ell_1}\cl(D_1\setminus A^1_j)$ and $F_2=\bigcap_{j=1}^{\ell_2}\cl(D_2\setminus A^2_j)$.

\begin{claim} \label{clm:discon1alt}  
$\cl(F_1\cup F_2)\subseteq F$. 
\end{claim}

\begin{proof}[Proof of Claim \ref{clm:discon1alt}]  Since $F$ is closed, it will suffice to show that $F_1\cup F_2\subseteq F$. Suppose $e\in F_1\cup F_2$. By symmetry we may assume that $e\in F_1$. Then $e\in \cl(D_1\setminus A^1_j)$ for all $1\leq i\leq \ell_1$. This implies that $e\in \cl(D\setminus A^1_j)$ for all $1\leq j\leq \ell_1$ since $D_1\subset D$, and $e\in \cl(D\setminus A^2_j)$ for all $1\leq j\leq \ell_2$ since $D_1\subset D\setminus A^2_j$. Hence $e\in F$ and we have $F_1\cup F_2\subseteq F$. This gives $\cl(F_1\cup F_2)\subseteq \cl(F)=F$.
\end{proof}

\begin{claim} \label{clm:discon2alt}  
$r(\cl(F_1\cup F_2))=r(F_1)+ r(F_2)$.
\end{claim}

\begin{proof}[Proof of Claim \ref{clm:discon2alt}]
Lemma \ref{lem:elemalt} implies that $r(\cl(D_1)\cup \cl(D_2))=r(\cl(D_1))+r(\cl(D_2))$.
 This tells us that the union of any two independent subsets of $\cl(D_1)$ and $\cl(D_2)$ is independent in $\cM$. Since $F_1\subseteq \cl(D_1)$ and  $F_2\subseteq \cl(D_2)$ this gives $r(F_1\cup F_2)=r(F_1)+ r(F_2)$. The claim now follows since $r(\cl(F_1\cup F_2))=r(F_1\cup F_2)$.
\end{proof}

Claims \ref{clm:discon1alt} and \ref{clm:discon2alt} immediately give
\begin{align*}
    r\left(\left(\bigcap_{j=1}^{\ell_1}\cl(D\setminus A^1_j)\right)\cap \left(\bigcap_{j=1}^{\ell_2}\cl(D\setminus A^2_j)\right)\right)
&\geq r\left(\bigcap_{j=1}^{\ell_1}\cl(D_1\setminus A^1_j)\right)+r\left(\bigcap_{j=1}^{\ell_2}\cl(D_2\setminus A^2_j)\right)\\
&= (\ell_1-k_1)+(\ell_2-k_2)
=\ell_1+\ell_2-k
\end{align*}
by Theorem \ref{prop:decompose}. We can combine this inequality with Theorem \ref{thm:strong} to deduce that $D$ is balanced.
\end{proof}

\begin{remark}
The converse of Lemma \ref{lem:disconalt} is false. Let $\cM$ be the cycle matroid of the graph $G=(V,E)$ with $V=\{u,x_1,x_2,x_3,v,y_1,y_2,w\}$ and $E=\{ux_1,x_1v,ux_2,x_2v,ux_3,x_3v,vy_1,y_1w,vy_2,y_2w,vw,uw\}$.
Then $D=E-uw$ is a disconnected 4-fold circuit in $\cM$ and we have $\cM|_D=\cM_{D_1}\oplus\cM|_{D_2}$ where $D_1=\{ux_1,x_1v,ux_2,x_2v,ux_3,x_3v\}$  and $D_2=\{vy_1,y_1w,vy_2,y_2w,vw\}$ are both 2-fold circuits. It is straightforward to check that $D$ and $D_2$ are balanced but $D_1$ is not.
\end{remark}

We can now immediately deduce the following.

\begin{theorem}\label{thm:disconalt} 
Suppose $D$ is a $k$-fold circuit in a matroid $\cM$ such that $\cM|_D = \bigoplus_{i=1}^s \cM|_{D_i}$.
If $D_i$ is a balanced $k_i$-fold circuit for all $1\leq i\leq s$, then $D$ is a balanced $k$-fold circuit.
\end{theorem}

\begin{proof}
This follows by recursively applying Lemma \ref{lem:disconalt} to $\cM|_D=\bigoplus_{i=1}^s \cM|_{D_i}$.    
\end{proof}

\begin{corollary}\label{cor:trivialstrongalt} For all $k\geq 1$, every trivial $k$-fold circuit in a matroid  is balanced.
\end{corollary}
\begin{proof}
This follows from Theorem \ref{thm:disconalt}, Lemma \ref{lem:trivialkcircuit} and the fact that all (1-fold) circuits of a matroid are balanced. 
\end{proof}

\section{Pseudomodular matroids}
\label{sec:pseudo}

Corollary~\ref{cor:cyclicflats} tells us that matroids whose lattice of flats is modular satisfy the $k$-fold circuit property.
However, Corollary~\ref{cor:balanced+modular+sublattice} shows this is a stronger condition than required, and that only a subset of the flats need to form modular pairs.
In studying the double circuit property, both \cite{HK} and \cite{DHK} showed the flats from Corollary~\ref{cor:balanced+modular+sublattice} form a modular sublattice when considering double circuits in certain families of `full' matroids.
We will generalise their results to $k$-fold circuits using a slight variant of the techniques developed in \cite{DHK}.

In general, finding a modular sublattice of a lattice is hard.
We next consider a class of lattices in which modular sublattices can be characterised by local conditions.

\begin{definition}
Suppose $\cL$ is a lattice and $r:\cL\to \ZZ$. We say that $r$ is \emph{pseudomodular on $\cL$} if for all $X,Y, Z \in \cL$ which satisfy
    \[
    r(X \vee Z) - r(X) = r(Y \vee Z) - r(Y) = r(X \vee Y \vee Z) - r(X \vee Y) \, ,
    \]
    we have
    \[
    r((X \vee Z) \wedge (Y \vee Z)) - r(X \wedge Y) = r(X \vee Z) - r(X) \, .
    \]
    The lattice $\cL$ is \emph{pseudomodular} if it is geometric and its grading function is pseudomodular. A matroid $\cM$ is  \emph{pseudomodular} if its lattice of flats is pseudomodular.
\end{definition}

Pseudomodularity can be viewed as a weakening of modularity in the following sense.
Modularity states that, for any $X,Y \in \cL$, we have
\[
r(X) - r(X \wedge Y) = r(X \vee Y) - r(Y) \, .
\]
This can informally be thought of as saying that opposite sides of the diamond in Figure~\ref{fig:pseudomodular} have equal length, and  this property holds for any diamond in $\cL$.
We can use a similar informal description of pseudomodularity.
Any choice of $X,Y,Z \in \cL$ generates a `diamond prism' as shown in Figure~\ref{fig:pseudomodular}.
Pseudomodularity states that if the three solid long `prism' edges have the same length, then the dotted edge must also have that length.

\begin{figure}
    \centering
    \includegraphics[width=0.25\linewidth]{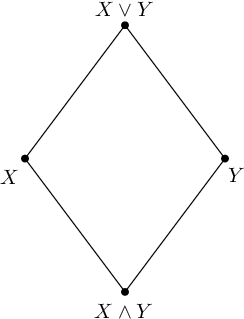} \qquad\qquad
    \includegraphics[width=0.5\linewidth]{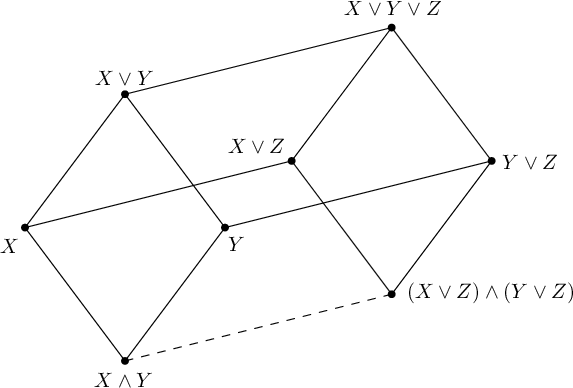}
    \caption{The left figure gives the `diamond' condition for modularity.
    The right figure gives the `diamond prism' condition for pseudomodularity.}
    \label{fig:pseudomodular}
\end{figure}

We will need the following elementary results on pseudomodularity.

\begin{lemma}\label{lem:pseudo}
    Suppose $\cL$ is a lattice and $r,s:\cL\to \ZZ$.
    \begin{enumerate}
    \item[(a)] If $r$ is modular then $r$ is pseudomodular.
    \item[(b)] If $r$ is semimodular and non-decreasing then $r(X \vee Y\vee Z) - r(X\vee Y) \leq r(X \vee Z) - r(X)$ for all  $X,Y,Z\in \cL$.
    \item[(c)] If $r,s$ are both semimodular,  non-decreasing and pseudomodular then $r+s$ is  pseudomodular.
    \end{enumerate}
\end{lemma}
\begin{proof} (a) Let $X,Y,Z \in \cL$ be such that 
\[
r(X \vee Z) - r(X) = r(Y \vee Z) - r(Y) = r(X \vee Y \vee Z) - r(X \vee Y) \, .
\] 
In particular, we note $r(X \vee Y) - r(X) = r(X \vee Y \vee Z) - r(X \vee Z)$.
From this and the modularity of $r$, we deduce that
\[
r(Y) - r(X \wedge Y) = r(X \vee Y) - r(X) = r(X\vee Y \vee Z) - r(X \vee Z) = r(Y \vee Z) - r((X \vee Z) \wedge (Y \vee Z)) \, .
\]
Hence $r((X \vee Z) \wedge (Y \vee Z)) - r(X \wedge Y) = r(Y \vee Z) - r(Y)$, and $r$ is pseudomodular.

(b) Since $r$ is semimodular and non-decreasing, we have 
$$r(X \vee Y) + r(X\vee Z)\geq r(X \vee Y\vee Z) + r((X\vee Y)\wedge (X\vee Z))\geq r(X \vee Y\vee Z) + r(X).$$
 This gives (b).

(c) Let $t=r+s$ and suppose that 
\[
    t(X \vee Z) - t(X) = t(Y \vee Z) - t(Y) = t(X \vee Y \vee Z) - t(X \vee Y) \, ,
\]
for some $X,Y,Z\in \cL$. Then
\[
    (r(X \vee Z) - r(X))+(s(X \vee Z) - s(X)) =   (r(X \vee Y \vee Z)) - r(X \vee Y))+(s(X \vee Y \vee Z) - s(X \vee Y)) \, .
\]
We can now apply (b) to deduce that $r(X \vee Z) - r(X)=r(X \vee Y \vee Z) - r(X \vee Y)$ and $s(X \vee Z) - s(X)=s(X \vee Y \vee Z)) - s(X \vee Y)$.
The symmetry between $X$ and $Y$ also gives $r(Y \vee Z) - r(Y)=r(X \vee Y \vee Z) - r(X \vee Y)$ and $s(Y \vee Z) - s(Y)=s(X \vee Y \vee Z) - s(X \vee Y)$. We can now use the pseudomodularity of $s,t$ to deduce that $r((X \vee Z) \wedge (Y \vee Z)) - r(X \wedge Y) = r(X \vee Z) - r(X)$ and $s((X \vee Z) \wedge (Y \vee Z)) - s(X \wedge Y) = s(X \vee Z) - s(X)$. Hence $t((X \vee Z) \wedge (Y \vee Z)) - t(X \wedge Y) = t(X \vee Z) - t(X)$. This completes the proof of (c).
\end{proof}

We will also need the following variant of \cite[Theorem 2.1]{DHK}.
{Its proof is very similar, but we include it for completeness.}
\begin{theorem}
\label{thm:modular+substructure}
    Let $\cL$ be a geometric lattice with grading $r$ and $(\cK, \rho)$ be a geometric modular lattice endowed with a non-decreasing modular function $\rho \colon \cK \rightarrow \ZZ_{\geq 0}$ (possibly distinct from its grading). Suppose $r$ is pseudomodular on a  sublattice $\cL'$ of $\cL$ and $\phi \colon \cK \rightarrow \cL'$ is such that
    \begin{enumerate}[label=(\roman*)]
        \item $\phi$ is $\wedge$-preserving, \label{eq:modular+i}
        \item $\phi(X \vee Y) = \phi(X) \vee \phi(Y)$ whenever $X,Y\in \cK$ and $X \vee Y = 1_\cK$, and \label{eq:modular+ii}
        \item $\rho(Y) = r(\phi(Y))$ whenever $Y\in \cK$ and  $[Y, 1_\cK]$ is an interval of length at most two. \label{eq:modular+iii}
    \end{enumerate}
    Then 
    \begin{enumerate}
    \item[(a)] $\phi$ is a lattice homomorphism which satisfies $\rho(X) = r(\phi(X))$ for all $X\in \cK$, and
   \item[(b)] $\phi(\cK)$ is a modular sublattice of $\cL$.
\end{enumerate}
\end{theorem}

\begin{proof}
    We first note that (a) implies (b) since, for any $\phi(X),\phi(Y) \in \phi(\cK)$, (a) will give
    \begin{align*}
    r(\phi(X)) - r(\phi(X) \wedge \phi(Y)) &= r(\phi(X)) - r(\phi(X \wedge Y)) = \rho(X) - \rho(X \wedge Y) 
    = \rho(X \vee Y) - \rho(Y)\\ &= r(\phi(X \vee Y)) - r(\phi(Y))  
    = r(\phi(X) \vee \phi(Y)) - r(\phi(Y)) \, .
    \end{align*}

    To prove (a), we show for any interval of $\cK$ of the form $I = [X,1_\cK]$, the induced map $\phi \colon I \rightarrow \cL'$ is a homomorphism satisfying $\rho(Y) = r(\phi(Y))$ for all $Y \in I$.
    The proof is by induction on $k$, the length of $I$.
    For the base case where $k=2$, the claim is immediate from conditions \ref{eq:modular+i}--\ref{eq:modular+iii}.

    For the general case, let $k \geq 3$.
    We first show that $\phi$ is $\vee$-preserving on $I$, i.e., for any $Y_1, Y_2 \in I$, we have $\phi(Y_1 \vee Y_2) = \phi(Y_1) \vee \phi(Y_2)$.
    We may assume that $Y_1 \wedge Y_2 = X$, else the claim holds by induction by considering the interval $[Y_1 \wedge Y_2, 1_\cK]$.
    If $Y_1 \vee Y_2 = 1_\cK$ then this is simply condition~\ref{eq:modular+ii}, so assume that $Y_1 \vee Y_2 < 1_\cK$.
    Let $Y_3 \in I$ be a \emph{complement} of $Y_1 \vee Y_2$, i.e., an element such that $Y_1 \vee Y_2 \vee Y_3 = 1_\cK$ and $X = (Y_1 \vee Y_2) \wedge Y_3$.
    The existence of such an element follows from \cite[Proposition 2.36]{Aigner} and the hypothesis that $\cK$ is a geometric lattice.

\begin{figure}[ht]
    \centering
    \includegraphics[width=0.45\linewidth]{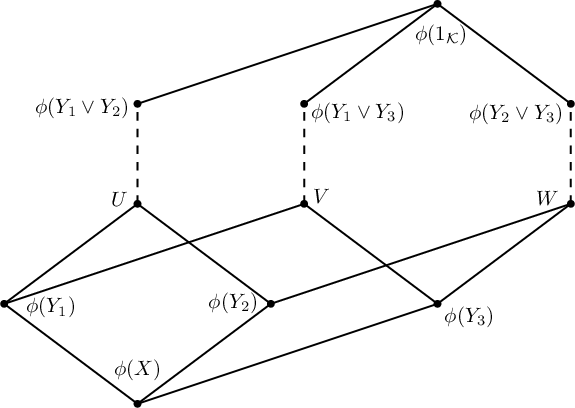}\qquad
    \includegraphics[width=0.45\linewidth]{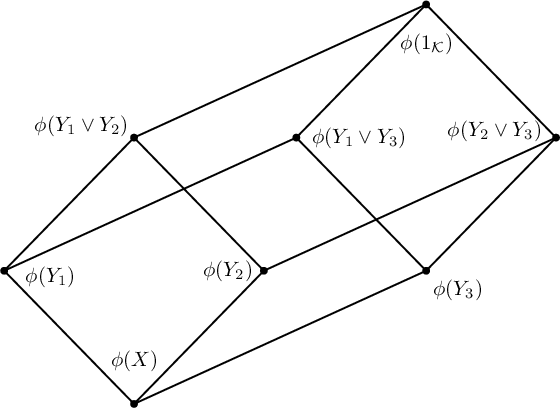}
    \caption{The posets analysed in the proof of Theorem~\ref{thm:modular+substructure}.}
    \label{fig:modular+proof}
\end{figure}
    
    Applying $\phi$ to $Y_1, Y_2, Y_3$ and their joins, we get the diagram given on the left of Figure~\ref{fig:modular+proof}, where $U = \phi(Y_1) \vee \phi(Y_2)$, $V = \phi(Y_1) \vee \phi(Y_3)$, and $W= \phi(Y_2) \vee \phi(Y_3)$.
    A useful observation for later is that by condition~\ref{eq:modular+ii}, we have $\phi(Y_1) \vee \phi(Y_2 \vee Y_3) = \phi(1_\cK)$, and hence
    \begin{equation} \label{eq:U+join}
    U \vee \phi(Y_2 \vee Y_3) = \phi(1_\cK) \, .    
    \end{equation}
    Furthermore, by condition~\ref{eq:modular+i}, $\phi$ is order-preserving and hence
    \begin{equation} \label{eq:phi+vee+hom}
    \phi(Y_1 \vee Y_2) \geq U = \phi(Y_1) \vee \phi(Y_2) \, .
    \end{equation}
    Similar inequalities holds for $V$ and $W$.
    We will show that equality holds in each of these inequalities, and hence we are in the situation of the diagram on the right of Figure~\ref{fig:modular+proof}.
    This will imply, in particular, that $\phi$ is $\vee$-preserving.
    
    Let $s \colon \cK \rightarrow \ZZ$ be the rank function of $\cK$.
    As $\cK$ is modular, $s$ is a modular function and hence:
    \begin{align*}
    s((Y_1 \vee Y_2) \wedge (Y_2 \vee Y_3)) &= s(Y_1 \vee Y_2) + s(Y_2 \vee Y_3) - s(1_\cK) \\
    &= s(Y_1 \vee Y_2) +  s(Y_2) + s(Y_3) - s(X) - s(1_\cK) = s(Y_2) \, .
    \end{align*}
    As $s$ is strictly increasing, it follows that $(Y_1 \vee Y_2) \wedge (Y_2 \vee Y_3) = Y_2$.
    By the induction hypothesis, we know that $\phi$ restricted to $[Y_2, 1_\cK]$ is a lattice homomorphism
    satisfying $\rho(Y') = r(\phi(Y'))$ for all $Y' \in [Y_2, 1_\cK]$.
    We can now use the modularity of $\rho$ to deduce that
    \begin{equation} \label{eq:pseudo+eq+1}
    r(\phi(Y_1 \vee Y_2)) + r(\phi(Y_2 \vee Y_3)) = r(\phi(1_\cK)) + r(\phi(Y_2)) \, .
    \end{equation}
    The semimodularity of $r$, \eqref{eq:U+join} and the fact that $r(U \wedge \phi(Y_2 \vee Y_3)) \geq r(\phi(Y_2))$ now  give
    \begin{align*}
        r(U) &\geq r(U \vee \phi(Y_2 \vee Y_3)) + r(U \wedge \phi(Y_2 \vee Y_3)) - r(\phi(Y_2 \vee Y_3)) \\
             &\geq r(\phi(1_\cK)) + r(\phi(Y_2)) - r(\phi(Y_2 \vee Y_3)) \\
             &= r(\phi(Y_1 \vee Y_2))\, ,
    \end{align*}
    where the final equality follows from \eqref{eq:pseudo+eq+1}.
    As $r$ is strictly increasing, this implies $U \geq \phi(Y_1 \vee Y_2)$.
    Coupled with \eqref{eq:phi+vee+hom}, we get $\phi(Y_1 \vee Y_2) = \phi(Y_1) \vee \phi(Y_2)$.

    It remains to show that $\rho(Y) = r(\phi(Y))$ for all $Y \in I$.
    If $Y \neq X$, this holds by induction on the interval $[Y, 1_\cK]$. Hence it only remains to show that $\rho(X) = r(\phi(X))$.
    Since the length of the interval $I = [X, 1_\cK]$ is at least 3, we can use the argument given in the proof  that $\phi$ is $\vee$-preserving on $I$ to find $Y_1, Y_2, Y_3 \in I \setminus X$ such that $Y_1 \wedge Y_2 = X$ and $Y_3$ is a complement of $Y_1 \vee Y_2$.
    This implies that $Y_1 \vee Y_2 \vee Y_3 = 1_\cK$ and 
    \begin{align*}
    \rho(Y_1) + \rho(Y_2) + \rho(Y_3) = \rho(Y_1 \vee Y_2) + \rho(Y_3) + \rho(X) = \rho(1_\cK) + 2 \rho(X) \, .
    \end{align*} 
    Note that, since $\phi$ is $\vee$-preserving on $I$, we are in the situation illustrated on the right of Figure~\ref{fig:modular+proof}.

    Now let $A = \phi(Y_1)$, $B = \phi(Y_2)$ and $C = \phi(Y_3)$.
By induction, $\phi$ restricted to $[Y_i, 1_\cK]$ is a lattice homomorphism
    satisfying $\rho(Y') = r(\phi(Y'))$ for all $Y' \in [Y_i, 1_\cK]$. Since $\rho$ is modular, this gives
    \[
    r(A \vee C) - r(A) = r(A \vee B \vee C) - r(A \vee B) = r(B \vee C) - r(B) = \rho(Y_3) - \rho(X) \, .
    \]
    As $\cL'$ is pseudomodular, we conclude that $r(\phi(Y_3)) - r(\phi(X))=\rho(Y_3) - \rho(X)$. Since $r(\phi(Y_3))=\rho(Y_3)$ this gives $ r(\phi(X))= \rho(X)$.
\end{proof}

\begin{theorem}\label{thm:pseudomodular+kCP}
    Every pseudomodular matroid has the $k$-fold circuit property for all $k\geq 1$.
\end{theorem}

\begin{proof}
    Let $\cM$ be a pseudomodular matroid, $\cL$ be the lattice of flats of $\cM$ and $D$ be a $k$-fold circuit of $\cM$ with principal partition $\{A_1, \dots, A_\ell\}$.
    Let $\cK$ be the Boolean lattice on all subsets of $L=\{1, \dots, \ell\}$ and define $\rho\colon \cK \rightarrow \ZZ_{\geq 0}$ by
    \[
    \rho(I) = |D| - k - \sum_{i\in L\setminus I} (|A_i| -1) = \sum_{i \in I} (|A_i| -1) + \ell - k \, , \quad I \subseteq L \, .
    \]
    Then $\cK$ is a geometric modular lattice whose grading is given by the cardinality of a subset, and $\rho$ is a non-decreasing modular function.
    Define a map $\phi\colon\cK\rightarrow \cL$ by $\phi(I) = \bigcap_{i\in L\setminus I} \cl(D \setminus A_i)$ for all proper subsets $I \subsetneq L$ and $\phi(L) = \cl(D)$.
    If $\phi$ satisfies conditions \ref{eq:modular+i}-\ref{eq:modular+iii} of Theorem~\ref{thm:modular+substructure}, then the sublattice $\phi(\cK)$ of $\cL$ is modular and  we can apply Corollary~\ref{cor:balanced+modular+sublattice} to deduce that $D$ is balanced.  
    It remains to show that $\phi$ satisfies conditions \ref{eq:modular+i}-\ref{eq:modular+iii} of Theorem~\ref{thm:modular+substructure}.

    To check condition \ref{eq:modular+i}, observe that
    \begin{align*}
        \phi(I \cap J) = \bigcap_{i \notin I \cap J} \cl(D \setminus A_i) = \left(\bigcap_{i \notin I} \cl(D \setminus A_i)\right) \cap \left(\bigcap_{j \notin J} \cl(D \setminus A_j)\right) = \phi(I) \wedge \phi(J) \, .
    \end{align*}
    
    For condition \ref{eq:modular+ii}, let $I \cup J = \{1, \dots, \ell\}$ and note that
    \begin{align*}
    \cl(D) \supseteq \phi(I) \vee \phi(J) &= \left(\bigcap_{i \notin I} \cl(D \setminus A_i)\right) \vee \left(\bigcap_{j \notin J} \cl(D \setminus A_j)\right) \\
    &\supseteq \left(\bigcap_{i \notin I} (D \setminus A_i)\right) \cup \left(\bigcap_{j \notin J} (D \setminus A_j)\right) \\
    &= \left(\bigcup_{i \in I} A_i\right) \cup \left(\bigcup_{j \in J} A_j\right) = D. \\
    \end{align*}
    Since $\phi(I) \vee \phi(J)$ is closed in $\cM$, this gives  $\phi(I) \vee \phi(J)= \cl(D)$.
    
    For condition~\ref{eq:modular+iii}, we have, by the definition of a $k$-fold circuit,
    \begin{align*}
        r(\cl(D)) &= |D| - k = \rho(\{1, \dots, \ell\}) \, , \mbox{ and} \\
        r(\cl(D \setminus A_i)) &= |D \setminus A_i| - (k-1)  = \rho(\{1, \dots, \ell\} \setminus \{i\}) \, .
    \end{align*}
    Let $A_i, A_j$ be distinct parts of the principal partition. By submodularity we have
    \begin{align}
    r(\cl(D \setminus A_i) \cap \cl(D \setminus A_j)) &\leq r(\cl(D \setminus A_i)) + r(\cl(D \setminus A_j)) - r(\cl(D \setminus A_i) \cup \cl(D \setminus A_j)) \nonumber \\ 
    &= |D \setminus A_i| - (k-1) + |D \setminus A_j| - (k-1) - |D| + k \nonumber \\
    &= |D| - k - (|A_i|-1) - (|A_j| -1) \nonumber \\
    &= \rho(\{1, \dots, \ell\} \setminus \{i,j\}) \, . \label{eq:2-intersection-closure}
    \end{align}
    Write $r^* = (r|_D)^*$ for the dual of the rank function restricted to $D$.
    By Lemma~\ref{lem:k-circuit+coflat} and submodularity, we have
    \[
    1 = r^*(A_i) < r^*(A_i \cup A_j)  \leq r^*(A_i) + r^*(A_j) = 2  ,
    \]
    and hence $r^*(A_i \cup A_j) = 2$.
    Then, by the dual rank formula, 
    \begin{align*}
        r(D \setminus (A_i \cup A_j)) &= r(D) - |A_i \cup A_j| + r^*(A_i \cup A_j) 
        \\ 
        &= |D| - k - (|A_i|-1) - (|A_j| -1) \\
        &= \rho(\{1, \dots, \ell\} \setminus \{i,j\}) \, .
    \end{align*}
    Hence $r(\cl(D \setminus A_i) \cap \cl(D \setminus A_j)) \geq \rho(\{1, \dots, \ell\} \setminus \{i,j\})$. Combined with \eqref{eq:2-intersection-closure}, this gives~\ref{eq:modular+iii}.
\end{proof}

We close this section by describing three examples of pseudomodular matroids from \cite{DL}.
\begin{itemize}
\item
The \emph{full cycle matroid} is the cycle matroid of a complete graph $K_n$. 
(This definition is slightly different to that given in \cite{DL}, which specifies that there should be infinitely many edges between each pair of vertices of the complete graph.)
As every graph on at most $n$ vertices can be embedded inside ${K}_n$, every cycle matroid of a graph can be embedded inside a full cycle matroid.

    \item 
Let $\FF$ and $\KK$ be two algebraically closed fields such that $\FF \subset \KK$ and $\KK$ has finite transcendence degree over $\FF$.
The \emph{full algebraic matroid} $\cA(\KK/\FF)$ has ground set $\KK$ whose independent sets are collections of elements that are algebraically independent over $\FF$.
The flats of this matroid are the algebraically closed subfields of $\KK$ that contain $\FF$.
As with linear matroids, any algebraic matroid can be embedded inside a full algebraic matroid, though the process is slightly more involved (see~\cite{DL}).

\item
Let $V$ be a finite set and let $E$ consist of countably infinitely many copies of each subset of $V$.
The \emph{full transversal matroid} has ground set $E$ whose independent sets are the subfamilies $X \subset E$ which have a system of distinct representatives.
Every transversal matroid on $V$ can be realised by a set system contained inside $E$, and hence every transversal matroid can be embedded inside some full transversal matroid.
\end{itemize}

\begin{corollary}\label{cor:full_alg_cycle}
    Full algebraic matroids, full cycle matroids and full transversal matroids satisfy the $k$-fold circuit property.
\end{corollary}

\begin{proof} 
It was shown in \cite{DL} that these matroids satisfy the weak series reduction property, and in \cite{Tan} that the weak series reduction property implies pseudomodularity. 
\end{proof}

Note that the proof in \cite{DL} that full cycle matroids satisfy the series reduction property holds for our modified definition of these matroids. 

\begin{remark}
In light of Theorem \ref{thm:pseudomodular+kCP}, one may wonder whether having the $k$-fold circuit property for all $k \geq 1$ is equivalent to being pseudomodular.
It turns out that the former is a strictly weaker property, as \cite{HK} give an example of a rank 5 matroid on 8 elements that is not pseudomodular but does satisfy the double circuit property.
It is straightforward to verify from their analysis that its lone $3$-fold circuit is balanced, hence this matroid satisfies the $3$-fold circuit property also.
As it has no $k$-fold circuits for $k >3$, this is trivially satisfied.
\end{remark}

\section{\boldmath The $k$-fold circuit property in count matroids}
\label{sec:countM}
    
    Makai~\cite{Mak} showed that the double circuit property holds for the count matroid $\cM(a,b)$ for all $0\leq b\leq 2a-1$. We will show that a $k$-fold circuit of $\cM(a,b)$ is balanced when all of its $(k-1)$-fold circuits are $\cM(a,b)$-rigid; see Proposition \ref{prop:balanced+Mab} below. Combined with Lemma \ref{lem:abclique}(b), this will give a relatively short proof of Makai's result. We will then give a partial extension of Makai's result by showing that $\cM(a,b)$ has the $k$-fold circuit property for all $k\geq 1$ whenever $0\leq b\leq a$.

\begin{lemma} \label{lem:count+kcircuit+rigidity}
    Let $G=(V,E)$ be a $k$-fold $\cM(a,b)$-circuit. Suppose that either $G$ is $\cM(a,b)$-connected or $b=0$. Then $G$ is $\cM(a,b)$-rigid.
\end{lemma}

\begin{proof}
    Suppose $G$ is $\cM(a,b)$-connected.
    We will show by induction on $k$ that $r(G) = a|V| - b$.
    The base case $k=1$ follows from Lemma \ref{lem:abclique}(b).

    For the induction step, by Proposition~\ref{prop:kcircuitears} we can write $G = G_1 \cup G_2$ where $G_1 = (V_1,E_1)$ is a $(k-1)$-fold $\cM(a,b)$-circuit and $G_2 = (V_2,E_2)$ is a 1-fold $\cM(a,b)$-circuit with $E_1 \cap E_2 \neq \emptyset$.
    By our induction hypothesis, both $G_1$ and $G_2$ are $\cM(a,b)$-rigid and hence 
    \[
    |E_1| = a|V_1| - b + (k-1) \, , \, |E_2| = a|V_2| - b + 1 \, .
    \]
    Suppose $G$ is not $\cM(a,b)$-rigid. Then $|E| = r(G) + k < a|V| - b + k$.
    This implies that
    \begin{align*}
    |E_1 \cap E_2| &= |E_1| + |E_2| - |E| \\
    &> (a|V_1| - b + (k-1)) + (a|V_2| - b + 1) - (a|V| - b + k) \\
    &= a|V_1 \cap V_2| - b \, .
    \end{align*}
    Hence $G_1 \cap G_2 = (V_1 \cap V_2, E_1 \cap E_2)$ is $\cM(a,b)$-dependent.
    As $G_1 \cap G_2 \subsetneq G_2$, this contradicts the fact that $G_2$ is an $\cM(a,b)$-circuit.

    It remains to consider the case when  $b=0$ and $G$ is not $\cM(a,0)$-connected. Let $G_1,G_2,\ldots,G_s$ be the subgraphs of $G$ induced by the connected components of $\cM(a,0)|_{E(G)}$. Then each $G_i$ is a $\cM(a,0)$-connected, $k_i$-fold circuit for some $k_i$ with $\sum_{i=1}^s k_i=k$ by Proposition~\ref{prop:decompose}, and hence
    each $G_i$ is $\cM(a,0)$-rigid by the first part of the proof.
    This implies that 
    $$r_{a,0}(G) = r_{a,0}(G_1) + \cdots + r_{a,0}(G_s) 
    = a|V_1| + \cdots + a|V_s| = a|V|
    $$
and hence $G$ is $\cM(a,b)$-rigid.
\end{proof}

We now turn to balanced $k$-fold circuits.
Let $G = (V,E)$ be a $k$-fold $\cM(a,b)$-circuit  with principal partition $\cA = \{A_1, \dots, A_\ell\}$.
We can consider the principal partition as a colouring of the edges of $G$.
We say a vertex $x\in V$ is \emph{monochromatic} if all edges of $G$ incident with $x$ belong to the same part $A_i$, and otherwise that $x$ is \emph{technicolour}.

The following result is a useful observation when studying balanced $k$-fold circuits in $\cM(a,b)$.

\begin{lemma}\label{lem:tech}
Let $G=(V,E)$ be a $k$-fold $\cM(a,b)$-circuit with principal partition $\{A_1,A_2,\dots, A_\ell\}$ and let $X$ denote the set of technicolour vertices in $G$.
Then $\sum_{i=1}^\ell |V(A_i)|=(\ell-1)|V|+|X|$ and $X$ is the vertex set of $\bigcap_{i=1}^\ell \cl(G - A_i)$. 
\end{lemma}

\begin{proof}
For any $(k-1)$-fold $\cM(a,b)$-circuit $G - A_i$ of $G$, we have $X \subseteq V(G - A_i)$.
Moreover, if $v \in V$ is monochromatic with all incident edges belonging to $A_j$, we have $v \notin V(G - A_j)$.
Moreover, since every $\cM(a,b)$-circuit has minimum degree at least $a+1\geq 2$ by Lemma \ref{lem:abclique}(b), we also have $v \notin V(\cl(G- A_j))$.  
This immediately gives the stated equality and also implies that  $X$ is the vertex set of $\bigcap_{i=1}^\ell \cl(G - A_i)$.    
\end{proof}
    
\begin{proposition} \label{prop:balanced+Mab}
    Suppose that $G$ is a $k$-fold $\cM(a,b)$-circuit  and that every $(k-1)$-fold $\cM(a,b)$-circuit in $G$ is $\cM(a,b)$-rigid.
    Then $G$ is a balanced $k$-fold $\cM(a,b)$-circuit.
\end{proposition}

\begin{proof}
Let $H$ be the underlying complete multigraph of $\cM(a,b)$.
We let $\{A_1,\dots,A_\ell\}$ be the principal partition of $G$ and put
$G-A_i=G_i=(V_i,E_i)$ 
for $1\leq i \leq \ell$.
By the hypothesis that every $(k-1)$-fold $\cM(a,b)$-circuit is $\cM(a,b)$-rigid, we have 
$|E_i|=a|V_i|-b+k-1$, and Lemma \ref{lem:abclique}(c) implies the $\cM(a,b)$-closure of $G_i$ is the induced subgraph $H[V_i]$.
Since the set $X$ of technicolour vertices of $G$ is the vertex set of $\bigcap_{i=1}^\ell(\cl(G_i))$ by Lemma~\ref{lem:tech}, this implies that it will suffice to show that
\begin{equation} \label{eq:HX+tight}
\ell-k= r_{a,b}\left(\bigcap_{i=1}^\ell \cl(G_i)\right) = r_{a,b}(H[X]) = \max(a|X| - b, 0) \, ,
\end{equation}
where the final equality follows from the fact that all induced subgraphs of $H$ are $\cM(a,b)$-rigid.

We first consider the case when $G$ is $\cM(a,b)$-rigid. 
The principal partition tells us every edge of $G$ is contained in exactly $\ell-1$, $(k-1)$-fold $\cM(a,b)$-circuits of $G$ and hence $\sum_{i=1}^\ell |E_i| = (\ell -1) |E|$.
Moreover, the assumption that $G$ is a $\cM(a,b)$-rigid $k$-fold circuit implies $|E| = a|V| -b + k$.
Combining these facts gives
$$\sum_{i=1}^\ell (a|V_i| - b + k-1) = \sum_{i=1}^\ell |E_i| = (\ell -1) |E| = (\ell -1)(a|V| - b + k) 
= a(\ell-1)|V| - (\ell -1)(b-k)$$
and hence  $a \sum_{i=1}^\ell |V_i| = a(\ell-1)|V| +b + \ell - k$.
We can now use the fact that $\sum_{i=1}^\ell |V_i| = (\ell -1)|V| + |X|$ by Lemma \ref{lem:tech} to deduce that $a|X| - b = \ell - k$.
This gives \eqref{eq:HX+tight} since $\ell \geq k$ by Proposition \ref{prop:principal+partition+k}.

It remains to consider the case when $G$ is not $\cM(a,b)$-rigid. Then, by Lemma~\ref{lem:count+kcircuit+rigidity}, $G$ is not $\cM(a,b)$-connected and $b \neq 0$.
Let $G_1,G_2,\ldots,G_s$ be the subgraphs of $G$ induced by the connected components of $\cM(a,b)|_{E(G)}$.
By Proposition~\ref{prop:decompose} and Lemma~\ref{lem:count+kcircuit+rigidity}, each $G_i$ is an $\cM(a,b)$-rigid, $k_i$-fold circuit where $k_i \leq k$. Theorem \ref{thm:disconalt} implies that it will suffice to show that each $G_i$ is balanced. This follows by the previous case.
\end{proof}

Since every $\cM(a,b)$-circuit is $\cM(a,b)$-rigid by Lemma \ref{lem:abclique}(b), we immediately deduce Makai's result.

\begin{corollary}
 $\cM(a,b)$ has the double circuit property for all $0\leq b\leq 2a-1$. 
\end{corollary}

We cannot use Proposition \ref{prop:balanced+Mab} to deduce that $\cM(a,b)$ has the $k$-fold circuit property for all $k$
because a $k$-fold circuit of $\cM(a,b)$ may contain $(k-1)$-fold circuits which are not $\cM(a,b)$-rigid when $k\geq 3$.
The case when $b=0$ is an exception to this since Lemma~\ref{lem:count+kcircuit+rigidity} implies that $k$-fold circuits are $\cM(a,0)$-rigid for all $k\geq 1$. Hence Proposition~\ref{prop:balanced+Mab} has the following immediate corollary.

\begin{corollary}\label{cor:b=0case_kfold}
    $\cM(a,0)$ has the $k$-fold circuit property for all $a,k\geq 1$.
\end{corollary}

We next show that $\cM(a,b)$ has the $k$-fold circuit property for all $k\geq 1$ when $1\leq b\leq a$.
We need the following structural result.

\begin{lemma}\label{lem:flats} Suppose  $a,b$ are integers with $1\leq b\leq a$. Then:
\begin{enumerate}
\item[(a)] every $(a,b)$-clique is $\cM(a,b)$-rigid;
\item[(b)] the union of two intersecting $\cM(a,b)$-rigid graphs is $\cM(a,b)$-rigid;
\item[(c)] every cyclic flat in $\cM(a,b)$ is a union of vertex disjoint $(a,b)$-cliques.
\end{enumerate}
\end{lemma}

\begin{proof}
(a) Lemma \ref{lem:abclique}(a) implies that the addition of a vertex incident to $a$ edges, including at most $a-b$ loops increases the rank of a graph in $\cM(a,b)$ by $a$. A simple induction based on the fact that the $(a,b)$-clique on one vertex is $\cM(a,b)$-rigid now gives (a).

(b) Since the $\cM(a,b)$-closure of every $\cM(a,b)$-rigid graph is an $\cM(a,b)$-clique by Lemma \ref{lem:abclique}(c), it will suffice to show that the union of two intersecting $\cM(a,b)$-cliques $G_1,G_2$ is $\cM(a,b)$-rigid. This follows by the same argument as in the proof of (a) since we can construct a spanning subgraph of $G_1\cup G_2$ from $G_1$ by recursively adding a vertex incident to $a$ edges, at most $a-b$ of which are loops.

(c) Let $G$ be a graph which is induced by a cyclic flat of $\cM(a,b)$.
Lemma \ref{lem:count+kcircuit+rigidity} and Lemma \ref{lem:abclique}(c) imply that every $\cM(a,b)$-connected component of $G$ is an $(a,b)$-clique. Hence it will suffice to show that the $\cM(a,b)$-connected components of $G$ are vertex disjoint. 
We will verify this by showing that the graph of every $\cM(a,b)$-circuit is connected. Suppose, for a contradiction that $H=(V,E)$ is a disconnected  $\cM(a,b)$-circuit.
Then we can choose two vertex disjoint subgraphs $H_1=(V_1,E_1)$ and $H_2=(V_2,E_2)$ of $H$ such that $H=H_1\cup H_2$. Since $H$ is an $\cM(a,b)$-circuit,   $H_1,H_2$ are $\cM(a,b)$-independent and hence $|E_i|\leq a|V_i|-b$ for $i=1,2$. This gives $|E|\leq a|V|-2b$ which contradicts the fact that $H$ is $\cM(a,b)$-rigid by Lemma \ref{lem:count+kcircuit+rigidity}.
\end{proof}

\begin{theorem}\label{thm:flats}
    $\cM(a,b)$ has the $k$-fold circuit property for all $k\geq 1$ and all $1\leq b\leq a$.
\end{theorem}

\begin{proof}
 Let $\cL$ be the lattice of flats of $\cM(a,b)$ and $\cL'$ be the set of flats in $\cL$ which are vertex disjoint unions of $(a,b)$-cliques. 
 Then $\cL'$ is a sublattice of $\cL$ since it is closed under the meet and join operations on $\cL$. 
 This follows since, for any $G_1, G_2 \in \cL'$, $G_1 \cap G_2$ is also a vertex disjoint union of $(a,b)$-cliques,
 and Lemma~\ref{lem:flats}(b) implies that $G_1 \cup G_2$ is a vertex disjoint union of $\cM(a,b)$-rigid graphs, so its closure is a vertex disjoint union of $(a,b)$-cliques by Lemma~\ref{lem:abclique}(c).
 
 We next show that $\cL'$ is pseudomodular.
 Consider the lattice of flats $\cL''$ of the full cycle matroid $\cM(1,1)$. It is well known that each flat in $\cL''$ is a vertex disjoint union of cliques and hence there exists a natural bijection $\theta:\cL''\to \cL'$  which maps a union of vertex disjoint cliques onto the corresponding union of vertex disjoint $(a,b)$-cliques. In addition,
 we saw in the proof of Corollary \ref{cor:full_alg_cycle} that the full cycle matroid is pseudomodular and hence $r_{1,1}$ is pseudomodular on $\cL''$. 
 Let $\nu:\cL''\to \ZZ$ be defined by  $\nu(F)= |V(F)|$ for all $F\in\cL''$. Then $\nu$ is modular and hence $(a-b)\nu+br_{1,1}$ is pseudomodular on $\cL''$ by Lemma \ref{lem:pseudo}.
 We will use this to deduce that $r_{a,b}$ is pseudomodular on $\cL'$ by
 showing  that 
 \begin{equation}\label{eq:rank}
 \mbox{$r_{a,b}(\theta (F))=(a-b)\nu(F)+br_{1,1}(F)$ for all $F\in \cL''$.} 
 \end{equation}
    
    Let $F\in \cL''$. Then $F$ is a vertex disjoint union of cliques $F_1, \dots, F_k$ and $\theta(F)$ is the vertex disjoint union of $(a,b)$-cliques $\theta(F_1), \dots, \theta(F_k)$.
Lemmas \ref{lem:abclique}(b) and \ref{lem:flats}(a) now give
    \[
    r_{a,b}(\theta(F)) = \sum_{i=1}^k r_{a,b}(\theta(F_i)) =\sum_{i=1}^k \left(a\nu(\theta(F_i))-b\right) =\sum_{i=1}^k \left((a-b) \nu(F_i) + br_{1,1}(F_i)\right) = (a-b)\nu(F) + br_{1,1}(F) \, .
    \]
Hence \eqref{eq:rank} holds and $r_{a,b}$ is pseudomodular on $\cL'$.

Let $D$ be a $k$-fold circuit in $\cM(a,b)$, $\{A_1,\ldots,A_\ell\}$ be the principal partition of $D$ and $L=\{1,2,\ldots,\ell\}$. Then $\cl(D)\in \cL'$ and $\cl(D\setminus A_i)\in \cL'$ for all $1\leq i\leq \ell$ by Lemma \ref{lem:flats}.
This implies that $\bigcap_{i \in L\setminus I} \cl(D \setminus A_i)\in \cL'$ for all proper subsets $I \subsetneq L$. We can now define $\cK$ and $\phi:\cK\to \cL'$ as in the proof of Theorem \ref{thm:pseudomodular+kCP} and apply Theorem \ref{thm:modular+substructure} to deduce that $\phi(\cK)$ is a modular sublattice of $\cL$. Corollary~\ref{cor:balanced+modular+sublattice} then tells us that $D$ is balanced.
\end{proof}

\section{Concluding remarks}
\label{sec:conclusion}

\noindent 1. It will be challenging to extend the proof technique of Theorem \ref{thm:flats} to the interval $a+1\leq b\leq 2a-1$ as $(a,b)$-cliques do not satisfy the properties of Lemma~\ref{lem:flats} when $a,b$ are in this range.
In particular, the structure of the cyclic flats of $\cM(a,b)$ becomes significantly more complicated.
Even the special case of $\cM(2,3)$-matroids seems to be difficult.
In this case, every cyclic flat  is an edge-disjoint union of cliques of size at least four intersecting in at most one vertex, and no union of two or more cliques is 3-connected.
As a first step, one could try to show the sublattice $\cL'$ of $\cL$ containing all cyclic flats of $\cM(2,3)$ is pseudomodular.
But even if we could do this, it would still not be sufficient to show $\cM(2,3)$ has the $k$-fold circuit property as the intersection of two cyclic flats may not be cyclic: the intersection is an edge-disjoint union of cliques intersecting in at most one vertex, but some of its cliques may have size two or three.
Thus we would have to consider a larger sublattice of $\cL$  than $\cL'$, and it is not obvious how this larger sublattice should be defined. \\

\noindent 2.
 {Rigidity matroids} are a family of matroids arising from discrete geometry which are closely related to count matroids.
In particular, the $d$-dimensional rigidity matroid  is equal to the count matroid $\cM(d,2d-1)$ when $d=1,2$ and hence  satisfies the double circuit property by \cite{Mak}. Characterising  the $d$-dimensional rigidity matroid is an important open problem when $d\geq 3$ and results on the structure of double circuits in these matroids would be useful in making progress on this problem. 
We will investigate this avenue of research in a forthcoming paper~\cite{JNS}. \\

\noindent 3. We have established that many examples of matroids which satisfy the double circuit property, in fact satisfy the $k$-fold circuit property for all $k \geq 1$.
Indeed, we do not know of any example of a matroid that satisfies the double circuit property but not the $k$-fold circuit property for some $k > 2$.
As such, we pose the following:
\begin{question}
    Does every matroid which satisfies the double circuit property, satisfy the $k$-fold circuit property for all $k\geq 1$?
\end{question}
A positive answer to this question would make verifying the $k$-fold circuit property easier, as double circuits have a simpler structure than $k$-fold circuits in general, particularly as $k$-fold circuits can be non-trivial but disconnected. \\

\noindent 4. Dress and Lov\'asz's initial interest in double circuits was to study the matroid matching problem.
We briefly describe a generalisation of this called the \emph{$k$-uniform matroid matching problem}.
Let $\cM=(E,r)$ be a matroid and $\cH$ be a set of rank $k$ flats of $\cM$.
A set of flats $H \subseteq \cH$ is a \emph{$k$-uniform matching} if $r(\bigcup_{p \in H} p) = k|H|$.
The $k$-uniform matroid matching problem is to compute a $k$-uniform matching of $\cH$ of maximum size, which we denote $\nu_k(\cH)$.
(Thus, the usual matroid matching problem is the special case when $k=2$.)
The upper bound on the maximum number of independent lines in projective spaces obtained by Lov\'asz in [12, Part I of the proof of Theorem 2] can be extended
to show that
\begin{equation} \label{eq:k+matching+minmax}
    \nu_k(\cH) \leq \min \left(r(Z)+\sum_{i=1}^t \left \lfloor \frac{r(Z \cup \bigcup_{p \in \cH_i} p) - r(Z)}{k}\right \rfloor \right)
\end{equation}
where the minimum is taken over all flats $Z \subseteq E$ of $\cM$ and for all partitions $\pi = (\cH_1,\cH_2,\dots, \cH_t)$ of $\cH$.

Dress and Lov\'asz \cite{DL} showed that \eqref{eq:k+matching+minmax} holds with equality when $k=2$ and $\cM$ has the double circuit property. It is unlikely, however, that the $k$-fold circuit property implies \eqref{eq:k+matching+minmax} holds with equality when $k\geq 3$, since we can use the NP-hardness of hypergraph matching (see~\cite{GJ}) to show that the  $k$-uniform matroid matching problem is NP-hard for representable matroids when $k\geq 3$. To see this, let $\Delta$ be a 3-uniform hypergraph  with $V(\Delta)=\{v_1,v_2,\dots,v_n\}$,
let $\{\e_1,\e_2,\ldots,\e_n\}$ be the standard basis for $\RR^n$, and  let $\cH=\{\langle\e_i,\e_j,\e_k\rangle : \{v_i,v_j,v_k\}\in E(\Delta)\}$.
Then the maximum size of a matching in $\Delta$ is equal to $\nu_3(\cH)$ over the full linear matroid $L_\RR^n$.

\section*{Acknowledgements}
A.\,N.\ and B.\,S.\ were partially supported by the Engineering and Physical Sciences Research Council grant EP/X036723/1.
The research visit that began this project was funded by the Heilbronn Institute for Mathematical Research small grant `Counting realisations of discrete structures'. For the purpose of open access, the author has applied a Creative Commons Attribution (CC-BY) licence to any Author Accepted Manuscript version arising.






\end{document}